%% file: constraints.tex
\documentclass[a4paper,11pt,DIV=11,
abstract=on,
parskip=half,
]{scrartcl}
 
\input{preamble}

\title{Generalized Iterative Scaling for Regularized Optimal Transport with Affine Constraints: Application Examples}
\author{
	Johannes von Lindheim\thanks{
		Institute of Mathematics,
		Technische Universit\"at Berlin,
		Strasse des 17. Juni 136, 10587 Berlin, Germany,
		vonlindheim@tu-berlin.de/steidl@math.tu-berlin.de
	},
	Gabriele Steidl\footnotemark[1]
}
\date{March 14, 2023}

\begin{document}

\maketitle

\begin{abstract}
We demonstrate the relevance of an algorithm called generalized iterative scaling (GIS) or simultaneous multiplicative algebraic reconstruction technique (SMART) and its rescaled block-iterative version (RBI-SMART) in the field of optimal transport (OT).
Many OT problems can be tackled through the use of entropic regularization by solving the Schr\"odinger problem, which is an information projection problem, that is, with respect to the Kullback--Leibler divergence.
Here we consider problems that have several affine constraints.
It is well-known that cyclic information projections onto the individual affine sets converge to the solution.
In practice, however, even these individual projections are not explicitly available in general.
In this paper, we exchange them for one GIS iteration.
If this is done for every affine set, we obtain RBI-SMART.
We provide a convergence proof using an interpretation of these iterations as two-step affine projections in an equivalent problem.
This is done in a slightly more general setting than RBI-SMART, since we use a mix of explicitly known information projections and GIS iterations.
We proceed to specialize this algorithm to several OT applications.
First, we find the measure that minimizes the regularized OT divergence to a given measure under moment constraints.
Second and third, the proposed framework yields an algorithm for solving a regularized martingale OT problem, as well as a relaxed version of the barycentric weak OT problem.
Finally, we show an approach from the literature for unbalanced OT problems.

\end{abstract}

\section{Introduction}

One of the key achievements in computational optimal transport (OT) is entropic regularization \cite{entropicReg12L}.
This modified problem can be solved by the celebrated Sinkhorn algorithm \cite{lightspeed13C}.
A large part of its success can likely be attributed to its remarkable simplicity.
Regularized OT is the so-called Schr\"odinger problem in a rewritten form, which is nothing more but the information projection problem
\begin{equation}\label{eq:schroedinger}
\min_{\pi\in \C} \KL(\pi, K),
\end{equation}
where $K_{ij}\coloneqq \exp(-c_{ij}/\eps)$ is the so-called Gibbs kernel, and $\C$ is the intersection of the affine subspaces encoding the marginal constraints of OT.
This problem is solved iteratively by performing alternating information projections onto the individual affine subspaces, which can be done in closed form and is also called iterative proportional fitting.
Subsequently, more Sinkhorn-like algorithms appeared for many other problems from OT, such as OT barycenters, multi-marginal OT, partial OT and more, since they are all just variants of the information projection above \cite{IBP15BCC}.
This strategy, however, is limited to constraints where the affine subspaces are axis-parallel hyperplanes or halfspaces, that is, the corresponding matrices are row-vectors containing only zeros and ones.
This is because the information projection to a general affine subspace cannot be explicitly computed, not even to one affine hyperplane.

Nonetheless, the projection to an affine subspace can be solved for iteratively using an algorithm called generalized iterative scaling (GIS) \cite{gis72DR} or simultaneous multiplicative algebraic reconstruction technique (SMART) \cite{smart93B}.
It has been related to the Sinkhorn algorithm \cite{toricMLE21ACKRS} and has been employed for finding OT plans with certain moment constraints, e.g. for solving unbalanced OT problems with non-exact marginal constraints \cite{toric22STVR}, see also Section~\ref{sec:unbalanced}.
In this paper, while cycling through our given affine subspaces, we perform only one iteration of this algorithm instead of a whole inner loop, which is more efficient.
If this is done with every subspace, this is the well-known rescaled block-iterative version RBI-SMART \cite{accelerating98B}.
Here we make a slight generalization that fits well into the context of OT: As one would expect, it is possible to mix the GIS steps towards the affine spaces with any explicitly known $\KL$ projections directly onto the affine space.
Our proof in a mild extension of an interpretation of the GIS iteration as a two-step $\KL$-projection in an equivalent problem \cite{geomintDR89C}.

We proceed to give several examples of OT problems where the framework above can be applied.
First, we consider the problem of minimizing the $\OT_\eps$ divergence to a given measure, subject to moment constraints, such as statistical or Fourier moments.
The algorithm specialized to this problem can be given in a simple dual form similar to the Sinkhorn algorithm.
This is more memory-efficient and uses only matrix-vector multiplications with the Gibbs kernel $K$, which can be computed using the fast Fourier transform.
Secondly, we consider the problem of martingale OT from mathematical finance, which is the standard OT problem with the additional affine constraint that the mean of the target locations of each source point should be the source point itself.
Our strategy yields a simple algorithm for this problem as well.
Third, we consider barycentric weak OT.
Here, the costs are a function of the source and mean target locations of each source point.
In order to apply our general algorithm, we first introduce a second auxiliary OT plan and relax the problem to a linear programming problem using Jensen's inequality.
We check that the error of this relaxation is controlled by the choice of resolution in the auxiliary plan.
The efficacy of each of these algorithms is illustrated numerically on toy data.
Finally, we consider an approach presented in \cite{toric22STVR} for unbalanced OT problems.


This paper is organized as follows:
In Section~\ref{sec:prelim_notation}, we fix our main notation and briefly introduce optimal transport.
In Section~\ref{sec:iproj_affine}, we consider the problem of projecting a probability distribution onto (intersections of) affine sets, introduce and motivate our main algorithm and prove its convergence.
Next, we apply it to several OT problems in Section~\ref{sec:applications}.
We conclude by mentioning more possible examples for applications and several directions for future research in Section~\ref{sec:conclusion}.


\section{Preliminaries and Notation}\label{sec:prelim_notation}

In this section, we fix our main notation and briefly introduce optimal transport.

\subsection{Notation}\label{sec:notation}

In the following, we denote by $\Vert \cdot \Vert$ the Euclidean norm on $\R^d$ and by
$\mathcal P(\R^d)$ the space of probability measures on $\mathbb R^d$.
The symbols $\smash{\R^M_{\geq 0}} = \{ x\in \R^M: x \geq 0 \}$ and $\smash{\R^M_{>0}} = \{ x\in \R^M: x > 0 \}$ are the nonnegative and positive orthant, respectively.
Let $\Delta_M = \{ x\in \smash{\R^M_{\geq 0}} : \smash{\sum_{j=1}^M} x_j = 1\}$ be the closed $(M-1)$-dimensional probability simplex.
We denote by $\zero_M\in \R^M$, $\zero_{M\times M}\in \R^{M\times M}$, $\one_M\in \R^M$ or $\one_{M\times M}\in \R^{M\times M}$ a vector or matrix of $M$ or $M\times M$ zeros and ones, respectively.
We may write $\zero$ or $\one$ if the dimension is clear from the context.
For $a, b\in \R^M$, we denote by $a \odot b \in \R^M$ element-wise multiplication and by $a/b=\frac a b \in \R^M$ element-wise division.
We define the Kullback--Leibler divergence $\KL\colon \R^M\times \R^M\to \R\cup \{ \infty \}$ as
\begin{equation}
\KL(p, q) \coloneqq \begin{cases}
\sum_{i=1}^M p_i\log \frac{p_i}{q_i} - p_i + q_i, &  0\leq p, q \text{ and } p\ll q \\
\infty, & \text{otherwise}
\end{cases}
\end{equation}
with the convention $0\log 0=0\log \frac 0 0 =0$.
We denote the row-wise vectorization of $\pi\in \R^{M\times M}$ by $\vec(\pi) = (\pi_{1, 1}, \pi_{1, 2}, \dots, \pi_{1, M}, \pi_{2, 1}, \dots, \pi_{M, M-1}, \pi_{M, M})\in \R^{M\cdot M}$.
We use the symbol ``$:$'' as an index to get a certain row or column of a matrix, for example $\pi_{i, :}\in \R^{1\times M}$ for the $i$-th row of $\pi\in \R^{M\times M}$ or $A_{:, j}\in \R^{m\times 1}$ for the $j$-th column of $A\in \R^{m\times M}$, respectively.


%
%
%

\subsection{Optimal Transport}

Let $1\leq p<\infty$.
Assume that we are given two discrete probability measures
\begin{equation}
	\mu = \sum_{i=1}^M \mu_i \delta (x_i), \quad 
	\nu = \sum_{j=1}^M \nu_j \delta (y_j), 
\end{equation}
where we abuse notation throughout this paper by identifying a probability measure with its weights.
Then the Monge--Kantorovich formulation of optimal transport is
\begin{equation}\label{eq:monge_kantorovich}
	\OT(\mu, \nu) 
	= \min_{\pi\in\Pi(\mu, \nu)} \sum_{i,j=1}^M c_{ij} \pi_{ij},
\end{equation}
where $c_{ij} = c(x_i, y_i)$ with $c\colon \R^d\times \R^d\to \R$ denotes some cost function and
\begin{equation}
	\Pi(\mu, \nu) \coloneqq \{ \pi\in \mathcal P(\R^d\times \R^d): \sum_{j=1}^M \pi_{ij} = \mu_i,  \sum_{i=1}^M \pi_{ij} = \nu_j \}
\end{equation}
is the set of probability measures on $\R^d \times \R^d$ with prescribed marginals $\mu$ and $\nu$.
When $c(x, y) = \Vert x-y\Vert^p$, then $\OT(\mu, \nu)=\W^p_p(\mu, \nu)$ defines the Wasserstein-$p$ distance $\W_p$.
The above optimization problem is convex, but can have multiple minimizers $\pi$.
The problem becomes strictly convex, guaranteeing a unique solution, if entropic regularization is applied:
\begin{equation}\label{eq:entropic_reg}
\OT_\eps(\mu, \nu) \coloneqq \min_{\pi\in \Pi(\mu, \nu)} \sum_{i,j=1}^M c_{ij} \pi_{ij} - \eps E(\pi),
\end{equation}
where
\begin{equation}
E(\pi)
= -\sum_{i, j=1}^M \pi_{ij} (\log(\pi_{ij}) - 1)
\end{equation}
with the convention $0\log 0=0$.
If $\hat\pi$ is a unique solution of \eqref{eq:monge_kantorovich} and $\hat\pi_\eps$ solves \eqref{eq:entropic_reg}, then $\hat\pi_\eps\to \hat\pi$ for $\eps\to 0$.
Entropic regularization is motivated by computational reasons as well.
Problem \eqref{eq:entropic_reg} can be rewritten as the Schr\"odinger problem \eqref{eq:schroedinger}, where $\C=\C^1\cap \C^2 \cap \Delta_{M\times M}$ with the two affine sets
\begin{equation}\label{eq:marg_constraints}
	\C^1 \coloneqq \{ \pi\in \R^M\times\R^M: \pi\one =\mu \}, \quad 
	\C^2 \coloneqq \{ \pi\in \R^M\times\R^M: \pi^\tt\one =\nu \}.
\end{equation}
This problem is generalized to many other OT applications such as Wasserstein barycenters, multi-marginal OT, OT with inequality constraints and more, see \cite{IBP15BCC}.
The solutions are always $\KL$-projections of some $K$ as in \eqref{eq:schroedinger}, but in possibly different domains than $\R^M\times\R^M$ and with other affine constraint sets $\C$.
So far, however, these affine constraints are usually limited to the case that they are defined by matrices containing only zeros and ones.
In the following section, we generalize this framework to affine spaces given by arbitrary matrices.


\section{Information Projection onto Affine Subspaces}\label{sec:iproj_affine}

Denote by
\begin{equation}
\C(A, b) \coloneqq \{ x\in \Delta_M : A x = b, A\in \R^{m\times M}, b\in \R^m \}
\end{equation}
an affine subset of the probability simplex.
Assume that we are given $n$ affine sets $\C^k\coloneqq \C(A^k, b^k)$, $A^k\in \R^{m_k\times M}, b^k\in \R^{m_k}$, $k=1, \dots, n$.
We set $\C\coloneqq \C^1\cap \dots \cap \C^n$.
For a given $q\in \R^M$, we consider the problem
\begin{equation}\label{eq:problem}
\min_{p\in \C} \KL(p, q).
\end{equation}
In what follows, we assume without loss of generality that $q>0$, since $q_j=0$ implies $p_j=0$ for any solution $p$ of \eqref{eq:problem}.
In the applications in Section~\ref{sec:applications}, we are mainly interested in solving some variants of the special case \eqref{eq:schroedinger}, where $q$ is the (vectorized) Gibbs kernel.

\subsection{Iterative Information Projections}\label{sec:iterative_iproj}

Since $\C^k$, $k=1, \dots, n$, are affine sets, it is well-known that problem~\eqref{eq:problem} can be solved by iterative information projections as follows:
\begin{equation}\label{eq:ibp_primal}
p^{(0)} \coloneqq q, \qquad
p^{(k)} \coloneqq P^\KL_{\C^k}(p^{(k-1)}) \quad \text{for all } k\in \N,
\end{equation}
where we extend our indexing periodically by setting $\C^{k+nl} \coloneqq \C^k$ for $k=1, \dots, n$, $l\in \N$, and similarly for $A^k$, $b^k$.

In the case that $A^k\in \{ 0, 1 \}^{1\times M}$, $b^k\in \R$, it is well-known, and not hard to check using the first-order optimality conditions, that $P^\KL_{\C^k}$ is given by scaling as follows:
\begin{equation}\label{eq:scaling}
	P^\KL_{\C^k}(p)_j
	= \begin{cases}
	p_j \cdot \frac{b^k}{A^k p}, & A_j = 1 \\
	p_j, & A_j = 0.
	\end{cases}
\end{equation}
This case occurs frequently in OT when $p$ is a measure on a product space and the $A^k$ are constraints on the marginals of $p$.

Unfortunately, there seems to be no closed-form expression for the projection to an affine subspace in general.
However, given $A^k$ and $b^k$, it can be computed iteratively using the GIS \cite{gis72DR} or SMART \cite{smart93B} algorithm and its rescaled block-iterative version RBI-SMART \cite{accelerating98B}, which we introduce in the subsequent section.
It has been related to the Sinkhorn algorithm \cite{toricMLE21ACKRS} and has been employed for finding OT plans with certain moment constraints, e.g. for solving unbalanced OT problems with non-exact marginal constraints \cite{toric22STVR}, see also Section~\ref{sec:unbalanced}.
Finally, we remark that if the constraint sets are not affine, but more general convex subsets with known information projections, the projection to their intersection can be determined with Dijkstra's algorithm \cite{dykstra83}.
This has been generalized to general Bregman projections, see \cite{dykConvergence00BL}.

\subsection{Generalized Iterative Scaling}\label{sec:gis}

We first consider the case where $n=1$.
Then given some positive vector $q\in \R^M_{>0}$, a matrix $A\in \R^{m\times M}$ and $b\in \R^m$, we are concerned with the following problem:
\begin{equation}\label{eq:ip_affine_problem}
\min_{p\in \Delta_M} \KL(p, q) \quad \text{such that}\quad Ap=b.
\end{equation}
Since the constraints are affine and $\KL$ is strictly convex in its first argument, if 
$\C(A, b)$ is nonempty, then there exists a unique minimizer.

GIS is a simple iterative procedure that can be viewed (i) as iterative information projections onto affine subsets \cite{geomintDR89C}, (ii) as a so-called mirror descent scheme for a particular objective and fixed step length \cite[Rem.~1]{mirrorDescent13PSBL}, or (iii) as an MM-algorithm on the dual formulation of \eqref{eq:ip_affine_problem}, see \cite{GISisMM19ST}.
In addition to providing an approximate solution of \eqref{eq:ip_affine_problem}, this algorithm is applicable to maximum-likelihood estimation for the exponential family \cite{geomintDR89C}.

Note that GIS is the more general SMART algorithm in our special case that no further regularization is applied, i.e., the case $\alpha=1$ in \cite{smart93B}.
Several other variants and extensions of SMART are also available, e.g. its rescaled, block-iterative version RBI-SMART \cite{accelerating98B} and BI-SMART for inconsistent affine constraints \cite{bismartinconsistent97B}, that is, when $\C=\emptyset$.
Another version of (RBI-)SMART with a weighting had been previously considered for the special case of entropy minimization in \cite{blockitentr87CS}.
An extension to more general Bregman divergences is discussed in \cite{blockitBregman02CH}.
For simplicity, we will mainly use the name GIS in what follows, since this corresponds to our special case of \eqref{eq:ip_affine_problem} with $\C\neq\emptyset$ and without further regularization.
%

In order to guarantee convergence, we assume that there exists a $p\in \Delta_M$ with $Ap=b$, and that
\begin{equation}\label{eq:gis_congerence_conditions}
	A, b \geq 0, \quad
	\sum_{i=1}^m A_{ij} = \one_m^\tt A = \one_M, \quad
	\sum_{i=1}^m b_i = \one_m^\tt b = 1.
\end{equation}
If the conditions \eqref{eq:gis_congerence_conditions} are not fulfilled, we can modify $A$ and $b$ so that these assumptions hold without altering the corresponding affine subspace according to the following steps:
\begin{enumerate}
	\item The domain constraint $p\in \Delta_M$ just means $p \geq 0$ together with the affine constraint $\one^\tt p = 1$.
	The former will automatically be satisfied since $\KL(p, q)=\infty$ if $p\not\in \smash{\R^M_{\geq 0}}$.
	Scaling the equations or adding the multiple of one constraint to another does not change the subspace.
	Thus, setting $\xi \coloneqq \min\{\min_{ij} \smash{A_{ij}}, \min_i \smash{b_i}\}$, modify
	\begin{equation}
	A^{(1)}\coloneqq A - \xi, \qquad
	b^{(1)}\coloneqq b - \xi,
	\end{equation}
	where the addition is meant element-wise.
	This ensures that $A^{(1)},$ $b^{(1)}\geq 0$ whilst retaining $\C(A, b) = \smash{\C(A^{(1)}, b^{(1)})}$.
	\item We set
	\begin{equation}
	\xi' \coloneqq \max\{ \max_j (\one_m^\tt A^{(1)})_j, \one_m^\tt b^{(1)} \}, \qquad
	A^{(2)} \coloneqq \frac{1}{\xi'}A^{(1)}, \quad b^{(2)} \coloneqq \frac{1}{\xi'}b^{(1)}.
	\end{equation}
	This guarantees that $\one_m^\tt A^{(2)},$ $\one_m^\tt b^{(2)}\leq 1$, and still $\C(A, b) = \smash{\C(A^{(2)}, b^{(2)})}$.
	\item Finally, adding or removing a linearly dependent constraint does not change the subspace either.
	We remove any rows $i$ for which $A^{(2)}_{ij}, b^{(2)}_i = 0$ for all $j=1, \dots, M$ and then set
	\begin{equation}
	A^{(3)} \coloneqq
	\begin{bmatrix}
	A^{(2)} \\
	1-\one_m^\tt A^{(2)}
	\end{bmatrix}, \quad
	b^{(3)} \coloneqq
	\begin{bmatrix}
	b^{(2)} \\
	1-\one_m^\tt b^{(2)}
	\end{bmatrix},
	\end{equation}
	This ensures $\one_m^\tt A = \one_M$ and $\one_m^\tt b = 1$, and still $\C(A, b) = \smash{\C(A^{(3)}, b^{(3)})}$.
\end{enumerate}
Finally, after these preparations of the linear system, we state the GIS algorithm:
\begin{align}\label{eq:gis}
p^{(0)} \coloneqq q, \qquad
p^{(k)} &\coloneqq p^{(k-1)}\odot \exp\Big(A^\tt \log\frac{b}{Ap^{(k-1)}}\Big)\\
&= p^{(k-1)}\odot \prod_{i=1}^m \Big( \frac{b}{Ap^{(k-1)}} \Big)^{A_{i, :}}, \quad k\in \N,
\end{align}
where $\exp$ and $\log$ are meant element-wise.
When the constraints are such that $A\in \{ 0, 1 \}^{1\times M}$, then this reduces to the projection by scaling in \eqref{eq:scaling}.
If $\C(A, b)\neq \emptyset$, then for $k\to \infty$, it holds that $\smash{\lim_{k\to \infty} p^{(k)} = P^\KL_{\C(A, b)}(q)}$, see, e.g., \cite{geomintDR89C,smart93B}.
This is mainly due to the following improvement inequality, which is easy to verify: For any $p\in \C(A, b)$, in particular for $p=\smash{P^\KL_{\C(A, b)}(q)}$, it holds that
\begin{equation}\label{eq:improvement_smart}
	\KL(p, p^{(k-1)}) - \KL(p, p^{(k)}) \geq \KL(b, Ap^{(k-1)}) \geq 0.
\end{equation}
This is called Fej\'er monotonicity of the sequence $(p^{(k)})_{k\in \N}$ with respect to $\C(A, b)$.

While choosing $\xi$, $\xi'$ with larger absolute value than necessary is always possible, we observe numerically that this slows down the convergence.
This is consistent with \eqref{eq:improvement_smart}, as
this decreases $\KL(b, Ap^{(k-1)})$.
Since the conditions \eqref{eq:gis_congerence_conditions} are not always necessary for convergence, it is in practice often even possible to obtain faster convergence using $\xi$, $\xi'$ with smaller absolute values than required to fulfill these conditions.

Instead of making $A$ column-stochastic by adding a row as in step $3$, another possibility is to scale the linear system appropriately \cite{blockitentr87CS}.
Since this introduces scaling factors into the iteration, we stick to the method above for convenience here.
We observed slower convergence numerically in some cases, but the impact of this choice on performance is not entirely clear.
We leave this matter for future research.

For an example of the algorithm's iterations, see Figure~\ref{fig:inf_proj_2d}.
\begin{figure}[htb]
	\centering
	\includegraphics[width=.55\textwidth]{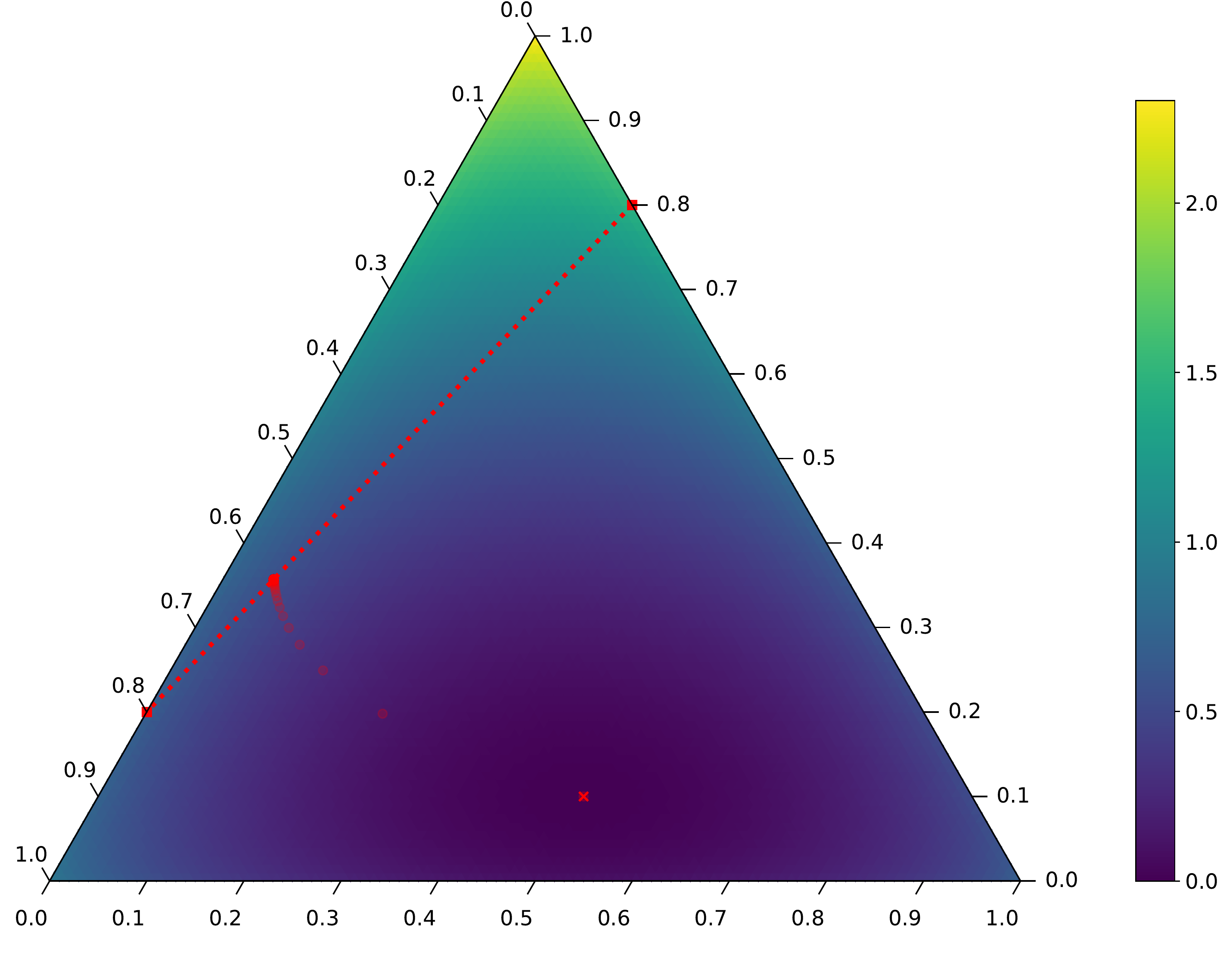}
	\caption{An illustration of $100$ iterations of GIS in \eqref{eq:gis} on the probability simplex with barycentric coordinates.
		In this example, we chose $M=3$ and $q=(0.5, 0.1, 0.4)$, which is indicated by the red ``x''.
		The red dotted line illustrates the affine constraint $\langle(0.1, 0.5, 0.4)^\tt, p\rangle = 0.42$, whereas the red transparent bullets indicate the trajectory of the iterations.
		The colormap indicates $\KL(\cdot, q)$.
		Note that the constraint that $p\in \Delta_M$ is only approximately fulfilled during the runtime of the algorithm, such that for this visualization, the displayed points are $p^{(k)}/\Vert p^{(k)}\Vert_1$, $k=0, \dots, 99$.}
	\label{fig:inf_proj_2d}
\end{figure}





\subsection{Swapping Projections for GIS Iterations}\label{sec:gis_swap}

We come back to the original problem~\eqref{eq:problem}, where $n$ may be larger than $1$.
We consider the situation that the projection $P^\KL_{\C^k}$ is unknown precisely for some $k\in I_\text{GIS}\subset \{ 1, \dots, n \}$.
Examples for this setting are given in Section~\ref{sec:applications}, where the known projections come from some marginal constraints and are of the form \eqref{eq:scaling}.
An ad hoc approach for this setting would be the following:
\begin{align}
	p^{(0)} &\coloneqq q \\
	p^{(k)} &\coloneqq P^\KL_{\C^k}(p^{(k-1)}) \quad \text{if } k \not\in I_\text{GIS}+n\N_0, \text{ and otherwise }\\
	p^{(k, 0)} &\coloneqq p^{(k-1)}, \quad
	p^{(k, t)} \coloneqq p^{(k, t-1)}\odot \exp(A^\tt \log(b / Ap^{(k, t-1)})) \text{ until converged in } t,
\end{align}
thus having an inner loop over $t$ for all $k\in I_\text{GIS}\subset \{ 1, \dots, n \}$.
If the $\smash{P^\KL_{\C^k}}$ are easy to compute, this is slow, since many iterations in the inner loop are having small effect if $p^{(k, t)}$ is close to convergence in $t$.
Instead, for $k \in I_\text{GIS}+n\N_0$, we propose to do only one GIS iteration in the inner loop, which leads to Algorithm~\ref{alg:gis_pkl_mix}.
\begin{algorithm}[htb]
	\begin{algorithmic}
		\State \textbf{Input:} $q\in \R^M_{>0}$, $I_\text{GIS}\subset \{ 1, \dots, n \}$, $A^k\in \R^{m_k\times M}$, $b^k\in \R^M$ for $k\in I_\text{GIS}$, $P^\KL_{\C^k}$ for $k\not\in I_\text{GIS}$
		\For{$k=1, \dots, N$}
			\State Normalize $A^k$, $b^k$ as described by the steps outlined in Section~\ref{sec:gis}
		\EndFor
		\State $p^{(0)}\coloneqq q$
		\For{$k=1, 2, \dots$}
		\State \begin{equation}
		p^{(k)} \coloneqq
		\begin{cases}
		p^{(k-1)}\odot\exp((A^k)^\tt \log \frac{b^k}{A^k p^{(k-1)}}), & k\in I_\text{GIS}+n\N_0 \\
		P^\KL_{\C^k}(p^{(k-1)}), & k\not\in I_\text{GIS}+n\N_0.
		\end{cases}
		\end{equation}
		\EndFor
		\caption{Iterative information projections with GIS iterations to compute $P^\KL_\C(q)$}
		\label{alg:gis_pkl_mix}
	\end{algorithmic}	
\end{algorithm}


If a GIS iteration is done in each step, that is, $I_\text{GIS} = \{1, \dots, n \}$, this algorithm is the RBI-SMART algorithm for \eqref{eq:problem}.
Thus, our algorithm is slightly more general, since we can use the exact projection instead of the iteration \eqref{eq:gis}, whenever it is available.
This makes more sense intuitively from a performance point of view.
While we only consider examples where the exact projections are iterative scaling and hence equal to GIS iterations as well, this is not always the case.
Relevant examples include Sinkhorn barycenters and OT with certain inequality constraints, see \cite[Secs.~3.2, 5.1, 5.3]{IBP15BCC}.

\subsection{Advantages of Block-Structure}\label{sec:block_advantages}

One natural question is in order: If all conditions are affine, what is the gain of considering multiple groups of constraints $A^1p=b^1, \dots, A^n p=b^n$ instead of stacking these constraints into one large system $Ap=b$?
First of all, in our considered examples from OT, the constraints come naturally in distinct groups already, and by merging these groups, one might lose explicit knowledge of the corresponding information projection and obtain a different, slower algorithm.
Secondly, the motivation of RBI-SMART is indeed to accelerate SMART \cite{accelerating98B}.

We illustrate this using the following example.
Consider the case of regularized OT \eqref{eq:schroedinger}, where $\C=\C^1\cap \C^2$, and $\C^1$, $\C^2$ are defined in \eqref{eq:marg_constraints}.
With $p=\vec(\pi)$, the corresponding linear systems are given by
\begin{align}
A^1
&\coloneqq \begin{bmatrix}
1 & \dots & 1 &&&&&&&\\
&&& 1 & \dots & 1 &&&& \\
&&&&&& \ddots &&& \\
&&&&&&& 1 & \dots & 1
\end{bmatrix}
= \id_M \otimes \one_M^\tt, \quad
b^1 \coloneqq \mu, \\
A^2 &\coloneqq \begin{bmatrix}
1 && & 1 && & & 1 && \\
& \ddots && & \ddots && \dots && \ddots & \\
&& 1 & && 1 & &&& 1
\end{bmatrix}
= \one_M^\tt \otimes \id_M, \quad
b^2 \coloneqq \nu.
\end{align}
In this way, Algorithm~\ref{alg:gis_pkl_mix} is the simple Sinkhorn algorithm: Written in terms of $\pi$, for $l\in\N$, we have the iterates
\begin{equation}\label{eq:sinkhorn}
	\pi^{(0)} \coloneqq \exp(-c/\eps),\qquad
	\pi^{(2l-1)} \coloneqq P^\KL_{\C^1}(\pi^{(2l-2)}),\qquad
	\pi^{(2l)} \coloneqq P^\KL_{\C^2}(\pi^{(2l-1)}),
\end{equation}
where $\smash{P^\KL_{\C^1}}$, $\smash{P^\KL_{\C^2}}$ are computed explicitly by scaling as
\begin{equation}\label{eq:sinkhorn_projections}
	P^\KL_{\C^1}(\pi^{(2l-1)})_{ij} = \pi^{(2l-2)}_{ij} \cdot \frac{\mu_i}{\sum_{s=1}^M\pi^{(2l-2)}_{is}}, \qquad
	P^\KL_{\C^2}(\pi^{(2l)})_{ij} = \pi^{(2l-1)}_{ij} \cdot \frac{\nu_j}{\sum_{s=1}^M\pi^{(2l-1)}_{sj}}.
\end{equation}
On the other hand, to the best of our knowledge, $\smash{P^\KL_{\C}}$ is unknown.
Using the GIS algorithm with $\C$ directly, we use the linear system defined by
\begin{equation}
	A \coloneqq \frac 1 2 \begin{bmatrix}
	A^1 \\
	A^2
	\end{bmatrix}, \quad
	b \coloneqq \frac 1 2 \begin{bmatrix}
	b^1 \\
	b^2
	\end{bmatrix},
\end{equation}
after a rescaling of the corresponding equations with $\xi'=2$ has been made as in Section~\ref{sec:gis} to obtain a column-stochastic constraint matrix, see \eqref{eq:gis_congerence_conditions}.
Let $\tilde p = \vec(\tilde\pi)$ denote the iterates of this algorithm, then
\begin{equation}\label{eq:sinkhorn_gis}
	\tilde\pi^{(0)} \coloneqq \exp(-c/\eps),\qquad
	\tilde\pi^{(l)}_{ij} \coloneqq \tilde\pi^{(l-1)}_{ij}\bigg( \frac{\mu_i}{\sum_{s=1}^M\tilde\pi^{(l-1)}_{is}}\cdot \frac{\nu_j}{\sum_{s=1}^M\tilde\pi^{(l-1)}_{sj}} \bigg)^{\frac 1 2}, \quad
	l\in \N.
\end{equation}
While the algorithms in \eqref{eq:sinkhorn} and \eqref{eq:sinkhorn_gis} are quite similar, we expect \eqref{eq:sinkhorn} to be faster than \eqref{eq:sinkhorn_gis} because we expect bigger update steps for two reasons:
\begin{enumerate}[(i)]
	\item the absence of square roots in \eqref{eq:sinkhorn_projections} resulting from the normalization with $\xi'=2$ and
	\item usage of more ``up-to-date information'' in each half iteration.
\end{enumerate}
Indeed, reason (i) is more formally explained in \cite{accelerating98B} using the improvement inequalities analogously to \eqref{eq:improvement_smart}, which also helps to better understand (ii).
For the Sinkhorn algorithm, we get for every $l\in \N_0$, $k=1, 2$ and $p\in \C$ an improvement in each half iteration by
\begin{equation}
	\KL(p, p^{(2l+k)}) - \KL(p, p^{(2l+k+1)}) \geq \KL(b^k, A^k p^{(2l+k)}),
\end{equation}
such that summing over $k$, we obtain for one full iteration cycle that
\begin{equation}\label{eq:sinkhorn_cycle_improvement}
\KL(p, p^{(2l-2)}) - \KL(p, p^{(2l)}) \geq \KL(b^1, A^1 p^{(2l-2)}) + \KL(b^2, A^2 p^{(2l-1)}).
\end{equation}
In comparison, the lower bound on the improvement for algorithm \eqref{eq:sinkhorn_gis} is only
\begin{equation}\label{eq:sinkhorn_gis_cycle_improvement}
	\KL(b, A\tilde p^{(l)})
	= \KL\Big(\frac 1 2 \begin{bmatrix}
	b^1 \\
	b^2
	\end{bmatrix}, \frac 1 2 \begin{bmatrix}
	A^1 \\
	A^2
	\end{bmatrix} \tilde p^{(l)}\Big)
	= \frac 1 2 (\KL(b^1, A^1 \tilde p^{(l)}) + \KL(b^2, A^2 \tilde p^{(l)})).
\end{equation}
The additional factor $\frac 1 2$ in \eqref{eq:sinkhorn_gis_cycle_improvement} corresponds to reason (i), see also \cite{accelerating98B} for an explanation why RBI-SMART is faster than BI-SMART.
Note that a similar argument can be made in problems where the positivity condition in \eqref{eq:gis_congerence_conditions} needs to be established by adding a positive constant $\xi$.
Being able to choose $\xi_k$ individually for each block will yield a larger right hand side in the improvement inequality.

However, we claim that this does not yet explain the improvement of \eqref{eq:sinkhorn} over \eqref{eq:sinkhorn_gis} fully.
In other words, BI-SMART, which converges since $\sum_{i,j=1}^M A_{ij}\leq 1$ is sufficient, already yields an improvement over SMART without renormalization.
Note that in \eqref{eq:sinkhorn_cycle_improvement}, each term also has the most recent iterate $p^{(2l-2)}$ and $p^{(2l-1)}$, respectively, whereas each term in \eqref{eq:sinkhorn_gis_cycle_improvement} has the same iterate $\tilde p^{(l)}$.
Since $p^{(2l-2)}$ was computed as a projection onto $\C^2$ without considering $\C^1$, intuitively, the constraint violation $\KL(b^1, A^1 p^{(2l-2)})$ should be larger than $\KL(b^1, A^1 \tilde p^{(l)})$, where $A^1$ enters in each iteration.
The same is true for the odd iterates.

In fact, preliminary numerical evidence suggests that \eqref{eq:sinkhorn_gis} requires approximately four times as many iterations to converge to a predefined level of accuracy with respect to the marginal constraints compared to \eqref{eq:sinkhorn}, independent of the regularization parameter $\eps$, the discretization dimension $M$ or the accuracy, for all problems we considered.

Another observation was that \eqref{eq:sinkhorn_gis} still converges, when the exponent $1/2$ is exchanged to $0.99$ (while $1$ does not work).
Consistent with our derivation above, the necessary number of iterations to reach a predefined level of accuracy is around half compared to previously.
We conclude that around half of the total improvement of the factor four can be attributed to each reason (i) and (ii), respectively.

\subsection{Convergence}

Having motivated Algorithm~\ref{alg:gis_pkl_mix}, we turn to showing its convergence.
Following the approach in \cite{geomintDR89C}, we will interpret the GIS iterations as two-step projections to affine sets, such that \cite[Thm.~3.2]{idivgeometry75C} applies.
Although our proof is quite similar to \cite{geomintDR89C}, the correct choice of the corresponding equivalent problem has to be done carefully in this more general setting.
\begin{theorem}\label{thm:gis_pkl_mix_convergence}
	Let $q\in \R^M_{>0}$, $\C=\cap_{k=1}^n \C^k$ and $p^{(k)}$ be defined as in Algorithm~\ref{alg:gis_pkl_mix} for all $k\in \N_0$.
	Then it holds
	\begin{equation}
	\lim_{k\to \infty} p^{(k)}
	= P^\KL_\C(q)
	= \argmin_{p\in \C} \KL(p, q).
	\end{equation}
\end{theorem}
\begin{proof}
	For $k=1, \dots, N$, Algorithm~\ref{alg:gis_pkl_mix} normalizes the $A^k\in \R^{m_k\times M}, b^k\in \R^{m_k}\geq 0$ so that
	\begin{equation}\label{eq:assumptions_general}
	\sum_{i=1}^{m_k} A^k_{i, j} = \one_{m_k}^\tt A = \one_M, \qquad
	\sum_{i=1}^{m_k} b^k_i = \one_{m_k}^\tt b = 1, \qquad
	k=1, \dots, n.
	\end{equation}
	For simplicity of notation, we will assume that $I_\text{GIS}=\{ n \}$, the other cases work similarly.
	We consider an alternative problem with $\tilde q_{i_1\dots i_n j}\coloneqq A^1_{i_1 j}\dots A^n_{i_n j}q_j$ and
	\begin{equation}
		\tilde\C^k \coloneqq \{ \tilde p\in \R^{m_1\times\dots\times m_n\times M}: \exists p\in \Delta_M, \text{ s.t. } \tilde p_{i_1 \dots i_n j} = A^1_{i_1 j}\dots A^n_{i_n j}p_j \text{ and } A^kp=b^k \},
	\end{equation}
	for $k=1, \dots, n-1$, and
	\begin{align}
		\tilde\C^{n-1/2} &\coloneqq \{\tilde p\in \R^{m_1\times\dots\times m_n\times M}: \sum_{i_1\dots i_{n-1} j}\tilde p_{i_1 \dots i_n j} = b^n_{i_n}, \quad i_n=1, \dots, m_n \}, \\
		\tilde\C^n &\coloneqq \{\tilde p\in \R^{m_1\times\dots\times m_n\times M}: \exists p\in \Delta_M \text{ s.t. } \tilde p_{i_1 \dots i_n j} = A^1_{i_1 j}\dots A^n_{i_n j}p_j \}.
	\end{align}
	Set $\tilde\C=\tilde\C^1\cap \dots \tilde\C^{n-1} \cap \tilde\C^{n-1/2} \cap \tilde\C^{n}$, then the problem we consider is
	\begin{equation}\label{eq:equivalent_problem}
		\min_{\tilde p\in \tilde\C} \KL(\tilde p, \tilde q).
	\end{equation}
	We show that \eqref{eq:equivalent_problem} is equivalent to \eqref{eq:problem}.
	To this end, consider $f\colon \R^M\to \R^{m_1\times\dots\times m_n\times M}$ and $F\colon \R^{m_1\times\dots\times m_n\times M}\to \R^M$ defined by
	\begin{equation}
	f(\tilde p)_j \coloneqq \sum_{i_1 \dots i_n} \tilde p_{i_1 \dots i_n j}, \qquad
	F(p)_{i_1 \dots i_n j} \coloneqq A^1_{i_1 j}\dots A^n_{i_n j}p_j.
	\end{equation}
	Using \eqref{eq:assumptions_general}, it is easy to check that these maps are bijections between $\Delta_M$ and $\tilde\C^n$ as well as between $\C^k$ and $\tilde\C^k$ for each $k=1, \dots, n-1$ with $f=F^{-1}$.
	Furthermore, for every $\tilde p$, $\tilde r$ with $\tilde p_{i_1 \dots i_n j}=A^1_{i_1 j}\dots A^n_{i_n j}p_j$ for some $p\in \R^M_{\geq 0}$ and likewise for $\tilde r$, it holds
	it holds that
	\begin{align}
	\KL(\tilde p, \tilde r)
	&= \sum_{i_1 \dots i_n j} \tilde p_{i_1 \dots i_n j} \log \frac{\tilde p_{i_1 \dots i_n j}}{\tilde r_{i_1 \dots i_n j}}
	- \sum_{i_1 \dots i_n j}\tilde p_{i_1 \dots i_n j}
	+ \sum_{i_1 \dots i_n j} \tilde r_{i_1 \dots i_n j} \\
	&= \sum_{i_1 \dots i_n j} p_jA^1_{i_1 j}\dots A^n_{i_n j} \log \frac{p_j}{r_j}
	- \sum_{i_1 \dots i_n j}p_jA^1_{i_1 j}\dots A^n_{i_n j}
	+ \sum_{i_1 \dots i_n j}r_jA^1_{i_1 j}\dots A^n_{i_n j} \\
	&= \sum_{j=1}^M p_j\log \frac{p_j}{r_j}
	- \sum_{j=1}^Mp_j
	+ \sum_{j=1}^Mr_j
	= \KL(p, r).
	\end{align}
	It follows for all $k=1, \dots, n-1$ that
	\begin{equation}\label{eq:correspondence}
		P^\KL_{\C^k} \circ f = f\circ P^\KL_{\tilde\C^k} \quad
		\Leftrightarrow \quad P^\KL_{\C^k} = f\circ P^\KL_{\tilde\C^k} \circ F.
	\end{equation}
	Next, corresponding to the iterates $p^{(k)}$, $k\in \N$ of Algorithm~\ref{alg:gis_pkl_mix}, set for $L\in \N_0$, $k\geq 1$,
	\begin{align}
	\tilde p^{(Ln+k)} &\coloneqq P^\KL_{\tilde\C^k}(\tilde p^{(Ln+k-1)}) \quad \text{for } k\in \{ 1, \dots, n-1\}, \\
	\tilde p^{(Ln-1/2)} &\coloneqq P^\KL_{\tilde\C^{n-1/2}}(\tilde p^{(Ln-1)}) \quad \text{for } L\geq 1, \\
	\tilde p^{(Ln)} &\coloneqq P^\KL_{\tilde\C^{n}}(\tilde p^{(Ln-1/2)}) \quad \text{for } L\geq 1.
	\end{align}
	Since $f(\tilde p^{(0)}) = f(\tilde q) = q = p^{(0)}$, assuming that $\tilde p^{(k-1)}=F(p^{(k-1)})$ for some $k\in \{ 0, \dots, n-1 \}$, we have by \eqref{eq:correspondence} that
	\begin{align}\label{eq:alg_correspondence}
		p^{(k)}
		= P^\KL_{\C^k}(p^{(k-1)})
		= f(P^\KL_{\tilde \C^k}(F(p^{(k-1)})))
		= f(P^\KL_{\tilde \C^k}(\tilde p^{(k-1)}))
		= f(\tilde p^{(k)})
	\end{align}
	or equivalently $\tilde p^{(k)} = F(p^{(k)})$ for all $k=0, \dots n-1$ by induction.
	We proceed to show that $\tilde p^{(n)}=F(c^{(n)}p^{(n)})$ as well for a normalization constant $c^{(n)}$.
	First, as in \eqref{eq:scaling}, the projection onto $\tilde\C^{n-1/2}$ is given by
	\begin{equation}
		\tilde p^{(n-1/2)}_{i_1\dots i_n j}
		= \tilde p^{(n-1)}_{i_1\dots i_n j}\cdot \frac{b^n_{i_n}}{\sum_{k_1\dots k_{n-1} l}\tilde p^{(n-1)}_{k_1\dots k_{n-1} i_n l}}.
	\end{equation}
	Since
	it holds that $\tilde p^{(n-1)}_{i_1\dots i_n j}=A^1_{i_1 j}\dots A^n_{i_n j}p^{(n-1)}_j$, we have
	\begin{equation}
	\tilde p^{(n-1/2)}_{i_1\dots i_n j}
	= A^1_{i_1 j}\dots A^n_{i_n j}p^{(n-1)}_j \frac{b^n_{i_n}}{\sum_{k_1\dots k_{n-1} l}A^1_{k_1 l}\dots A^n_{k_n l}A^n_{i_n l}p^{(n-1)}_l}
	= A^1_{i_1 j}\dots A^n_{i_n j}p^{(n-1)}_j \frac{b^n_{i_n}}{(A^n p^{(n-1)})_{i_n}
	}.
	\end{equation}
	Note that $\tilde\C^{n-1/2}\subset \Delta_{m_1\dots m_n M}$, since for $\tilde p^{(n-1/2)}\in \tilde\C^{n-1/2}$, we have
	\begin{equation}
		\sum_{i_1\dots i_n j}A^1_{i_1 j}\dots A^n_{i_n j}p^{(n-1)}_j \frac{b^n_{i_n}}{(A^n p^{(n-1)})_{i_n}}
		= \sum_{i_n=1}^{m_n} b^n_{i_n} \frac{\sum_{j=1}^M A^n_{i_n j} p^{(n-1)}_j}{\sum_{l=1}^M A^n_{i_n l} p^{(n-1)}_l}
		= \sum_{i_n=1}^{m_n} b^n_{i_n}
		= 1.
	\end{equation}
	Next, consider the projection onto $\tilde\C^n$.
	Recall that
	\begin{equation}
	p^{(n)}
	= p^{(n-1)}\odot \exp \Big( (A^n)^\tt \log \frac{b}{A^n p^{(n-1)}} \Big)
	\end{equation}
	and set $c^{(n)}\coloneqq (\sum_{j=1}^M p^{(n)}_j)^{-1}$.
	If $\tilde p\in \tilde\C^n$ with $\tilde p_{i_1\dots i_n j}=A^1_{i_1 j}\dots A^n_{i_n j}p_j$ for some $p\in \Delta_M$, then
	\begin{align}
		\KL(\tilde p, \tilde p^{(n-1/2)})
		&= \sum_{i_1\dots i_n j} \tilde p_{i_1\dots i_n j}\log\frac{\tilde p_{i_1\dots i_n j}}{\tilde p^{(n-1/2)}_{i_1\dots i_n j}}
		- \underbrace{\sum_{i_1\dots i_n j} \tilde p_{i_1\dots i_n j}}_1
		+ \underbrace{\sum_{i_1\dots i_n j} \tilde p^{(n-1/2)}_{i_1\dots i_n j}}_1 \\
		&= \sum_{i_1\dots i_n j} p_j A^1_{i_1 j}\dots A^n_{i_n j}\log \frac{p_j}{p^{(n-1)}_j \frac{b^n_{i_n}}{(A^n p^{(n-1)})_{i_n}}} \\
		&= \sum_{j=1}^M p_j\log p_j - \sum_{j=1}^M p_j\Big(\log p^{(n-1)}_j + \sum_{i_n=1}^{m_n} A^n_{i_n j} \log\Big( \frac{b}{A^n p^{(n-1)}} \Big)_{i_n} \Big) \\
		&= \KL(p, p^{(n)})
		= \KL(p, c^{(n)}p^{(n)}) + \log c^{(n)}.
	\end{align}
	This is minimal for $p=c^{(n)}p^{(n)}$, such that $\tilde p^{(n)}=F(c^{(n)}p^{(n)})$.
	It is easy to check that if $p\in \Delta_M$, then $\KL(p, cq)=\KL(p, q)+C$, where $C$ does not depend on $p$.
	Since $\C^{n+1}=\C^1\subset \Delta_M$, this implies
	\begin{equation}
		p^{(n+1)}
		= P^\KL_{\C^1}(p^{(n)})
		= P^\KL_{\C^1}(c^{(n)}p^{(n)}).
	\end{equation}
	Together with \eqref{eq:alg_correspondence}, it holds that
	\begin{align}
		&(\tilde p^{(0)}, \tilde p^{(1)}, \dots, \tilde p^{(n-1)}, \tilde p^{(n)}, \tilde p^{(n+1)}, \dots) \\
		={}& (F(p^{(0)}), F(p^{(1)}), \dots, F(p^{(n-1)}), F(c^{(n)}p^{(n)}), F(p^{(n+1)}), \dots).
	\end{align}
	By \cite[Thm.~3.2]{idivgeometry75C}, we have $\tilde p^{(k)}\to P^\KL(\tilde q)$ for $k\to \infty$.
	Furthermore, it follows from the proof of this reference that $\KL(\tilde p^{(n)}, \tilde p^{(n-1/2)})\to 0$.
	Plugging this into the calculation above, it follows $c^{(n)}\to 1$.
	Hence, we also have $ p^{(k)}\to P^\KL(q)$, which completes the proof.
\end{proof}

\section{Applications in Optimal Transport}\label{sec:applications}

In this section, we consider several applications of the approach presented in Section~\ref{sec:iproj_affine} in the field of optimal transport.
In what follows, we will always consider the discrete case where $\supp(\mu)$, $\supp(\nu) \subset \{ x_1, \dots, x_M \}$.
Note that we choose the common symbol $M$ for notational convenience.
Extending this to different support sets or sizes of $\mu$ and $\nu$ is straightforward.

\subsection{OT with Moment Constraints}\label{sec:min_problem_moments}

In this section, given a measure $\mu \in \mathcal P(\R^d)$, a set of $m$ ``test functions'' $A\in \R^{m\times M}$ and $b\in \R^m$, we consider problems of the form
\begin{equation}\label{eq:min_problem_moments}
\min_{\nu\in \C(A, b)} \OT_\eps(\mu, \nu).
\end{equation}
Of particular interest are constraints on the expectation value and variance, or the Fourier coefficients of $\nu$.
Problem \eqref{eq:min_problem_moments} can be rewritten as
\begin{align}\label{eq:min_problem_moments_plan}
\min_{\nu\in \C(A, b)} \OT_\eps(\mu, \nu)
= \min_{\nu\in \C(A, b)} \min_{\pi\in \Pi(\mu, \nu)} \sum_{i,j=1}^M c_{ij}\pi_{ij} - \eps E(\pi)
= \min_{\substack{\pi \one = \mu \\ A \pi^\tt \one = b}} \sum_{i,j=1}^M c_{ij}\pi_{ij} - \eps E(\pi).
\end{align}
In turn, denoting by $K\coloneqq \exp(-c/\eps)$ the so-called Gibbs kernel, this problem is equivalent to
\begin{equation}\label{eq:schroedinger_moments}
P^\KL_\C(K) \coloneqq \argmin_{\pi\in \C} \KL(\pi, K),
\end{equation}
where in our case, we have $\C = \C^1 \cap \C^2$ with
\begin{equation}
\C^1\coloneqq \{ \pi \in \R^{M\times M}: (P^1)_\# \pi = \mu \}, \qquad
\C^2\coloneqq \{ \pi \in \R^{M\times M}: (P^2)_\# \pi \in \C(A, b) \}.
\end{equation}
In what follows, we assume that $\mu>0$ in order to obtain positive values inside the logarithm in Algorithm~\ref{alg:gis_pkl_mix}.
This is without loss of generality, since for any $i$ with $\mu_i=0$, the solution of \eqref{eq:min_problem_moments} will have $\pi_{i,:}=0$.

The problem on the right hand side of \eqref{eq:schroedinger_moments} fits into the framework presented in Section~\ref{sec:iproj_affine} and Algorithm~\ref{alg:gis_pkl_mix} applies with the iterations
\begin{equation}\label{eq:PKL1_primal}
\pi^{(k)} \coloneqq P^\KL_{\C^1}(\pi) = \diag(\mu/ \pi\one)\pi^{(k-1)}
\end{equation}
for odd $k$, see \eqref{eq:scaling},
and
\begin{equation}\label{eq:PKL2_primal}
\pi^{(k)} \coloneqq \pi^{(k-1)}\diag(\exp(A^\tt \log(b / A (\pi^{(k-1)})^\tt \one))).
\end{equation}
for even $k$.
However, similar as with the Sinkhorn algorithm, we can derive the usual dual form for memory efficiency.

\subsubsection{Derivation of Dual Algorithm}\label{sec:dual_alg}

In this section, we show that as for the standard Sinkhorn algorithm, it is possible to recover the primal from the dual solution.
This has major benefits with respect to memory, since the number of dual variables to store is only $2M$ compared to $M^2$ primal variables, i.e., entries in the transport plan $\pi$.
Introducing Lagrangian multipliers to \eqref{eq:min_problem_moments_plan}, we get
\begin{equation}
	L(\pi, \alpha, \beta)
	= \sum_{i,j=1}^M (c_{ij}\pi_{ij} + \eps \pi_{ij}\log(\pi_{ij}) - \eps \pi_{ij}) - \alpha^\tt(\pi \one  - \mu) - \beta^\tt(A\pi^\tt\one - b).
\end{equation}
Thus, we get as optimality conditions for all $i, j=1, \dots, M$ that
\begin{equation}
	0 = c_{ij} + \eps \log(\pi_{ij}) - \alpha_i - \beta^\tt A_{:, j}.
\end{equation}
We rearrange and summarize this as
\begin{equation}\label{eq:primal_dual_relation}
	\pi
	= \diag(\exp(\alpha/\eps))\exp(-c/\eps) \diag(\exp(A^\tt \beta / \eps))
	= \diag(u) K \diag(v),
\end{equation}
where we have substituted the scaling variables $u=\exp(\alpha/\eps)$, $v=\exp(A^\tt \beta/\eps)$.
Plugging in the constraint $\pi\one=\mu$ into \eqref{eq:primal_dual_relation} yields
\begin{align}
	\mu
	= \pi\one
	= \diag(u)K\diag(v)\one
	= \diag(u)Kv
	= u\odot Kv,
\end{align}
which rearranges to the well-known Sinkhorn iteration
\begin{equation}\label{eq:PKL1_dual}
	u^{(k)} \coloneqq \mu/Kv^{(k-1)}, \qquad v^{(k)} \coloneqq v^{(k-1)}.
\end{equation}
In fact, this corresponds to the projection \eqref{eq:PKL1_primal}.
On the other hand, plugging in \eqref{eq:PKL2_primal} into \eqref{eq:primal_dual_relation} yields
\begin{align}
\diag(u^{(k)})K\diag(v^{(k)})
&\coloneqq \diag(u^{(k-1)})K\diag(v^{(k-1)})\diag(\exp(A^\tt \log(b / A (\pi^{(k-1)})^\tt \one))) \\
&= \diag(u^{(k-1)})K\diag(v^{(k-1)} \odot \exp(A^\tt \log(b / A (\pi^{(k-1)})^\tt \one))),
\end{align}
such that we only update the scaling variable $v$ as
\begin{align}\label{eq:PKL2_dual}
v^{(k)}
&= v^{(k-1)} \odot \exp(A^\tt \log(b / A (\diag(u^{(k-1)})K\diag(v^{(k-1)}))^\tt \one)) \nonumber \\
&= v^{(k-1)} \odot \exp(A^\tt \log(b / A (v^{(k-1)}\odot K^\tt u^{(k-1)}))).
\end{align}
Note that it is never necessary to store the transport plan $\pi\in \R^{M\times M}$, since
\begin{equation}
	\nu=\diag(v)K^\tt \diag(u)\one = v\odot K^\tt u.
\end{equation}
In the case when $c_{ij}= \Vert x_i-x_j\Vert^2$, even the multiplication with $K$, which is a Gaussian convolution, can be carried out without allocating $O(M^2)$ memory for $K\in \R^{M\times M}$ using fast Fourier transforms, see, e.g., \cite{FFTboost22LPP}.
While this is also possible for non-equidistant grids using the non-equispaced Fourier transform (NFFT), this simplifies for equidistant grids, where a convolution can be performed using the conventional fast Fourier transform (FFT) with $O(M\log M)$ arithmetical operations.
This exploits that $K$ is a Toeplitz matrix.
Since we stay in this simpler setting using the squared Euclidean distance $c(x_i, x_j)=\Vert x_i-x_j\Vert^2$, we briefly outline the necessary computations:
Denoting by $\FFT_M$, $\IFFT_M$ the FFT, respectively inverse fast Fourier transform (IFFT) of length $M$, for
\begin{equation}
	s \coloneqq (0, \Vert x_1-x_2\Vert^2, \Vert x_1-x_3\Vert^2, \dots, \Vert x_1-x_M\Vert^2, 0, \Vert x_1-x_M\Vert^2, \dots, \Vert x_1-x_2\Vert^2)^\tt,
\end{equation}
and $L\coloneqq \FFT_{2M} (\exp(-s/\eps))$, it holds
\begin{equation}
	Ka = \IFFT_{2M}(L\odot \FFT_{2M}((a, \zero_M)^\tt))
\end{equation}
For the cyclical convolution on the torus, for $t\coloneqq (0, \Vert x_1-x_2\Vert^2, \Vert x_1-x_3\Vert^2, \dots, \Vert x_1-x_M\Vert^2)^\tt$ and $L' \coloneqq \FFT_M(\exp(-t/\eps))$, this simplifies further to
\begin{equation}
Ka = \IFFT_M(L'\odot \FFT_M(a)),
\end{equation}
since it is no longer required to embed $K$ into a circulant matrix.

We summarize the derivation above in Algorithm~\ref{alg:sinkhorn}.
Convergence is clear by Theorem~\ref{thm:gis_pkl_mix_convergence}, since Algorithm~\ref{alg:sinkhorn} is just a special case of Algorithm~\ref{alg:gis_pkl_mix} written in dual form.

\begin{algorithm}[htb]
	\begin{algorithmic}
		\State \textbf{Input:} $\mu\in \R^M$, $c\in \R^{M\times M}$, $\eps >0$, $A\in \R^{m\times M}$, $b\in \R^m$
		\State Normalize $A$, $b$ 
		as described by the steps outlined in Section~\ref{sec:gis}
		\State $K\gets \exp(-c/\eps)$
		\State $u, v \gets \one_M$
		\While{not converged}
			\State $u\gets \mu/Kv$
			\State $v\gets v \odot \exp(A^\tt \log(b / A (v\odot K^\tt u)))$
		\EndWhile
		\State \textbf{Output:} $\nu \coloneqq v\odot K^\tt u\in \R^M$
		\caption{Sinkhorn/GIS algorithm for $\OT_\eps$-minimization with moment constraints}
		\label{alg:sinkhorn}
	\end{algorithmic}	
\end{algorithm}

We briefly comment on an alternative approach to \eqref{eq:PKL2_dual}.
Plugging $A\pi^\tt\one = b$ into \eqref{eq:primal_dual_relation} yields the constraints
\begin{align}
	b
	= A\pi^\tt \one
	= A\diag(v)K^\tt \diag(u)\one
	= A(v\odot K^\tt u)
	= A(\exp(A^\tt \beta / \eps) \odot K^\tt u).
\end{align}
This nonlinear system of equations in $\beta$ can be solved using a Newton-scheme.
Similar as above, we find numerically that alternatingly performing \eqref{eq:PKL1_dual} and one Newton-iteration yields a convergent algorithm.
We do not discuss this approach further, since each Newton-iteration requires the solution of a linear system and convergence is not clear, without apparent benefits of this approach over the other.

\subsubsection{Numerical Examples}

We present two proof-of-concept examples: Constraining mean and variance of a measure supported on the unit interval, and constraining the mean of a measure supported on the torus.

For the first example, we divide the unit interval into the uniform grid
\begin{equation}
0 = x_1 < x_2 < \dots < x_M = 1, \qquad x_i = \frac{i-1}{M-1}, \quad i=1, \dots, M
\end{equation}
with $M=100$.
Let $x=(x_1, \dots, x_M)$.
Furthermore, $\mu= \sum_{i=1}^M \mu_i \delta_{x_i}$, where we sample $\mu_i$ from the probability density function (PDF) of the normal distribution $\mathcal N(0.4, 0.1^2)$ and normalize $\mu$ to sum to one.
We take $c(x_i, x_j) = \vert x_i-x_j\vert^2$ as the cost function and $\eps=0.01$.
Then we solve \eqref{eq:min_problem_moments} using Algorithm~\ref{alg:sinkhorn}, where for $X\sim \mu$, $Y\sim \nu$, we pose the constraints that
\begin{equation}\label{eq:constr_unit_interval}
\E[Y] 
= \E[X] + 0.1
= 0.5 \quad \text{and} \quad 
\Var[Y] 
= 1.5^2\Var[X] 
= 0.15^2.
\end{equation}
The expectation constraint can be expressed as $x^\tt \nu = \E[Y] = 0.5$.
Moreover, since
\begin{equation}
0.15^2
= \Var[Y]
= \E[Y^2] - \E[Y]^2
= \E[Y^2] - 0.5^2,
\end{equation}
the variance constraint can be expressed as $(x\odot x)^\tt \nu = 0.5^2 + 0.15^2$.
Thus, we solve \eqref{eq:min_problem_moments} with
\begin{equation}
A = \begin{bmatrix}
x_1 & \dots & x_M \\
x_1^2 & \dots & x_M^2
\end{bmatrix},
\qquad
b = \begin{bmatrix}
0.5 \\
0.5^2+0.15^2
\end{bmatrix}
\end{equation}
using Algorithm~\ref{alg:sinkhorn}.
For comparison, we use an alternative approach without entropic regularization.
This can be done by solving the following linear program (LP):
\begin{equation}
\min_{\pi\in \R^{M\times M}} \vec(c)^\tt \vec(\pi)
\quad
\text{such that}
\quad
\begin{bmatrix}
\one e_1^\tt & \one e_2^\tt & \dots & \one e_M^\tt \\
A & A & \dots & A
\end{bmatrix}
\vec(\pi)
=
\begin{bmatrix}
\mu \\
b
\end{bmatrix}.
\end{equation}
The results are displayed in Figure~\ref{fig:gaussian01}.
While the linear program solution also fulfills the constraints, we observe severe undesirable grid noise artifacts, while this solution was more expensive to compute.
On the other hand, the result of the Sinkhorn-algorithm is, as one might expect, another (cut-off) Gaussian with the given mean and variance.
It is hardly visible in Figure~\ref{fig:gaussian01}, since it matches the PDF of $\mathcal N(0.5, 0.15^2)$ very closely.
\begin{figure}[htb]
	\centering
	\includegraphics[width=.7\textwidth]{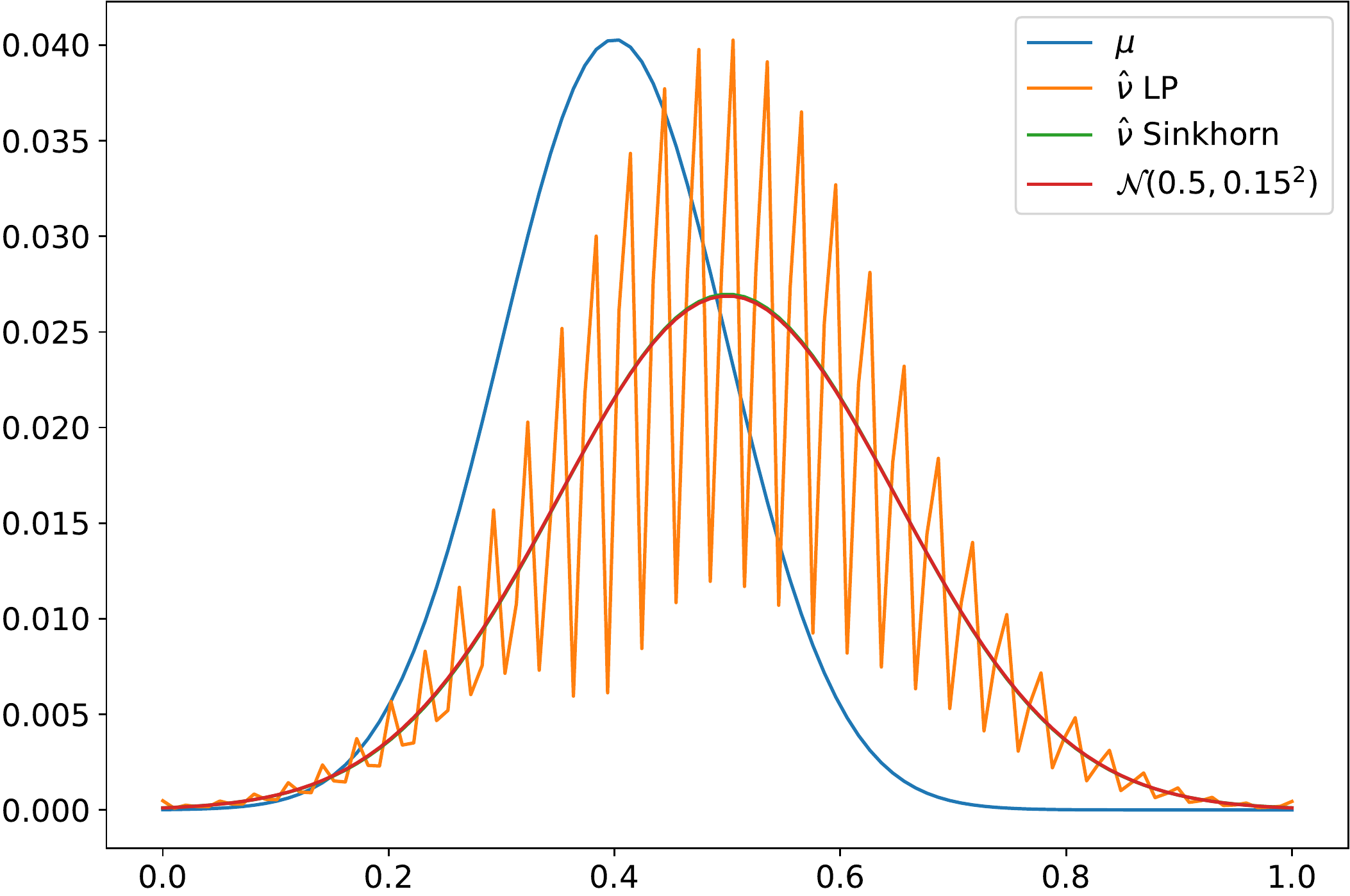}
	\caption{Results for problem \eqref{eq:min_problem_moments} with constraints \eqref{eq:constr_unit_interval}.
		Depicted are $\mu$, the solutions $\hat\nu$ computing using the Sinkhorn-like Algorithm~\ref{alg:sinkhorn} and using a linear program, respectively, and, for comparison, $\mathcal N(0.5, 0.15^2)$.}
	\label{fig:gaussian01}
\end{figure}

Next, we consider optimal transport on the torus and constrain the Fourier coefficients of the solution.
The von Mises distribution $\mathcal M(\gamma, \kappa)$ with mean $\gamma$ and concentration parameter $\kappa$ is an analog of the Gaussian distribution on the torus and is characterized by the PDF
\begin{equation}
f(x; \gamma, \kappa) = \frac 1 C \exp(\kappa \cos(x-\gamma)),
\end{equation}
where $C= \int_{-\pi}^\pi \exp(\kappa \cos(x-\gamma))$ is the normalizing constant.
We divide the torus into the uniform grid
\begin{equation}
-\pi = x_1 < x_2 < \dots < x_M = \pi - \frac{2\pi}{M}, \qquad
x_i = -\pi + 2\pi\frac{i-1}{M}, \quad i=1, \dots, M
\end{equation}
with $M=500$ and let again $x=(x_1, \dots, x_M)$.
As above, we construct $\mu$ by sampling this distribution with parameters $\gamma=-1$ and $\kappa=1/(0.2\pi)^2$ at the grid positions.
We denote by
\begin{equation}
z\coloneqq \E_\mu[\e^{\imag x}] = \sum_{i=1}^M \mu_i \e^{\imag x_i} \in \mathbb{C}
\end{equation}
the circular mean or circular first moment of $\mu$.
Note that this is the first non-trivial Fourier coefficient of $\mu$.
For this example, we want to constrain this quantity of $\nu$ to be
\begin{equation}\label{eq:constr_torus}
\E_\nu[\e^{\imag x}] = \vert z \vert \e^{\imag (\arg(z) + \pi/2)} \eqqcolon z',
\end{equation}
that is, we want the circular mean of $\nu$ to be rotated by a quarter of the unit circle compared to $\mu$.
For solving this problem using Algorithm~\ref{alg:sinkhorn}, we convert this constraint to real-valued constraints as
\begin{align}
\sum_{i=1}^M \cos(x_i)\mu_i = \Real(z'), \qquad
\sum_{i=1}^M \sin(x_i)\mu_i = \Imag(z'),
\end{align}
such that
\begin{equation}
A = \begin{bmatrix}
\cos(x_1) & \cos(x_2) & \dots & \cos(x_M) \\
\sin(x_1) & \sin(x_2) & \dots & \sin(x_M)	
\end{bmatrix},
\qquad
b = \begin{bmatrix}
\Real(z') \\
\Imag(z')
\end{bmatrix}.
\end{equation}
The results are depicted in Figure~\ref{fig:vonmises} on the left, which also contains a von Mises distribution rotated by $\pi/2$ for comparison.
Interestingly, the solution $\hat\nu$ is something different: Because of the periodicity of the torus, a distribution with two modes has a lower cost.
An inspection of the corresponding transport plan displayed on the right hand side of Figure~\ref{fig:vonmises} reveals that the smaller bump of $\hat\nu$ indeed ``wraps around'' and approaches $\mu$ from the other side.
\begin{figure}[htb]
	\centering
	\includegraphics[width=\textwidth]{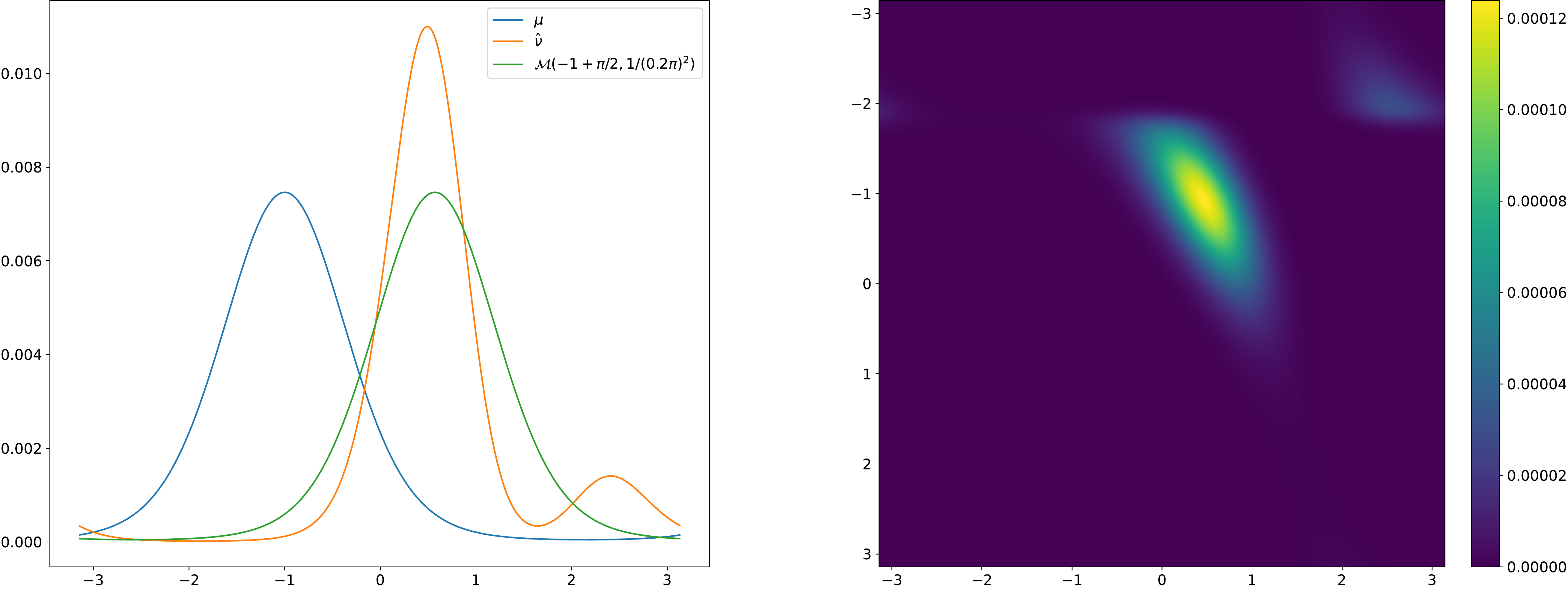}
	\caption{Results for problem \eqref{eq:min_problem_moments} with constraints \eqref{eq:constr_torus}.
	Left: $\mu$, the solution $\hat\nu$  and, for comparison, $\mathcal M(-1+\pi/2, 1/(0.2\pi)^2)$.
	Right: Corresponding transport plan $\pi\in \Pi(\mu, \hat\nu)$.}
	\label{fig:vonmises}
\end{figure}

\subsection{Martingale Optimal Transport}\label{sec:martOT}

Martingale OT is a variant of the standard Monge--Kantorovich formulation of optimal transport \eqref{eq:monge_kantorovich} with the additional constraint that if $(X, Y)\sim \pi\in \Pi(\mu, \nu)$, then we require for the conditional expectation that $\E[Y\vert X] = X$.
This problem comes from mathematical finance \cite{martOT13BHP,martOT14GHT}.
For an introduction to and thorough analysis of martingale optimal transport, we refer to \cite{martOT16BJ}.
In the discrete setting with $\supp(\mu), \supp(\nu)\subset \{ x_i\in \R: i=1, \dots, M \}$, $x_1<\dots<x_M$, the problem thus reads
\begin{equation}\label{eq:martOT}
\min_{\pi\in \Pi(\mu, \nu)} \sum_{i,j=1}^M \pi_{ij}c_{ij} \quad \text{such that} \quad
x_i
= \frac{\sum_{j=1}^M \pi_{ij}x_j}{\sum_{j=1}^M \pi_{ij}}
= \frac 1 {\mu_i} \sum_{j=1}^M \pi_{ij}x_j,
\quad
i=1, \dots, M,
\end{equation}
where we again assume without loss of generality that $\mu>0$, or otherwise, the rows in $\pi$ corresponding to zero-entries in $\pi$ are zero as well.
Problem \eqref{eq:martOT} has a solution, if and only if $\mu$ and $\nu$ are in convex order \cite{martOTexistence65S}, denoted by $\mu \preceq \nu$, which means that
\begin{equation}
	\int \varphi \dd \mu \leq \int \varphi \dd \nu
\end{equation}
for all convex functions $\varphi\colon\R\to \R$, as well as both measures having finite mass and finite first moments.
Intuitively, it means that $\nu$ is ``more spread out'' than $\mu$.
This is perhaps more easily understood through the following characerization in dimension $d=1$.
Let
\begin{equation}
	F^{-1}_\mu(y) \coloneqq \inf\{ x\in \R: \mu((-\infty, x]) \geq y \}, \quad y\in (0, 1)
\end{equation}
denote the quantile function of a probability measure $\mu$.
Then $\mu\preceq\nu$, if and only if
\begin{equation}
	\int_0^y (F^{-1}_\mu(z) - F^{-1}_\nu(z)) \dd z \geq 0 \quad \text{for all }  y\in [0,1],
\end{equation}
with equality for $y=1$.
For more details, we refer to the recent paper \cite{convexorderchar22WZ} for an overview and a list of references on this notion.
The same authors show that, interestingly, convex order can also be characterized using the Wasserstein distance $\W_2^2$: It holds $\mu\preceq\nu$ precisely when
\begin{equation}
	\W_2^2(\nu, \rho) - \W_2^2(\mu, \rho) \leq \int\vert x\vert^2\dd\nu(x) - \int\vert x\vert^2\dd\mu(x)
\end{equation}
for all probability measures $\rho\in \mathcal P(\R^d)$ with bounded support.

\subsubsection{Algorithm Derivation}\label{sec:martOT_alg_derivation}

We derive a numerical method for \eqref{eq:martOT}.
For previous work in this direction, refer to \cite{entropicMartOT18dM,martOTmethods19GO}.
We approach this problem using entropic regularization and Algorithm~\ref{alg:gis_pkl_mix}.
This works similarly as for Monge--Kantorovich OT \eqref{eq:monge_kantorovich} and the Sinkhorn algorithm, except for the additional martingale constraint, for which we need to derive the corresponding GIS iteration.
For simplicity, we assume the one-dimensional case $d=1$, but this method is straightforward to generalize to arbitrary dimensions by treating all components independently.
Set $x=(x_1, \dots, x_M)\in \R^M$ as a row vector and for $k=1, \dots, M$,
\begin{equation}
	\C^k\coloneqq \{ \pi\in \R^{M\times M}: \frac{1}{\mu_k} \pi_{k, :}x^\tt=x_k \}.
\end{equation}
Similar as in Section~\ref{sec:min_problem_moments}, we get problem \eqref{eq:schroedinger_moments} for $\C\coloneqq \Pi(\mu,\nu)\cap \C^1\cap \dots\cap \C^M$.
Similarly as with the scaling \eqref{eq:scaling}, it is easy to check that $P^\KL_{\C^k}$ only updates the $k$-th row.
For this update, set $A^k\coloneqq x$, $b^k=x_k$.
Furthermore, set $\bar\pi\coloneqq \diag(\mu^{-1})\pi$ to be the row-normalized version of $\pi$, then the martingale OT constraint in \eqref{eq:martOT} reads as
\begin{equation}\label{eq:mart_constr_1}
	A^k \bar\pi^\tt_{k, :} = b^k \quad \text{for all } k=1, \dots, M.
\end{equation}
Let $\xi\coloneqq \min_j x_j$, $\xi'\coloneqq \max_j x_j-\xi$ and set
\begin{equation}\label{eq:x_normalized}
	A\coloneqq \begin{bmatrix}
	(x-\xi)/\xi' \\
	1-(x-\xi)/\xi'
	\end{bmatrix}
\end{equation}
for the normalization described in Section~\ref{sec:gis}.
Since $\min_j x_j\leq b^k\leq \max x_j$ and $A^1=\dots=A^M=x$, \eqref{eq:mart_constr_1} is equivalent to
\begin{equation}\label{eq:mart_constr_2}
	A\bar\pi_{k, :}^\tt = A_{:, k} \quad \text{for all } k=1, \dots, M,
\end{equation}
with corresponding GIS iteration
\begin{align}
	\bar\pi_{k, :} &\gets
	\bar\pi_{k, :}\odot \exp\Big( A^\tt \log \frac{A_{:, k}}{A\bar\pi^\tt_{k, :}}  \Big)^\tt
	= \frac 1 {\mu_k} \pi_{k, :}\odot \exp\Big( A^\tt \Big( \log \frac{A_{:, k}}{A\pi^\tt_{k, :}} + \log(\mu_k)\one_2 \Big) \Big)^\tt \\
	&= \frac 1 {\mu_k} \pi_{k, :}\odot \exp\Big( A^\tt \log \frac{A_{:, k}}{A\pi^\tt_{k, :}} \Big)^\tt \odot \exp(\log(\mu_k)\one_M)^\tt
	= \pi_{k, :}\odot \exp\Big( A^\tt \log \frac{A_{:, k}}{A\pi^\tt_{k, :}} \Big)^\tt.
\end{align}
Writing \eqref{eq:mart_constr_2} more compactly as
\begin{equation}\label{eq:mart_constr_3}
	A\bar\pi^\tt = A,
\end{equation}
we can do the projection to all $\C^k$ simultaneously by the update
\begin{equation}\label{eq:simultaneous_update}
	\pi\gets \diag(\mu)\pi\odot \exp\Big(A^\tt \log\frac{A}{A\pi^\tt}\Big)^\tt.
\end{equation}
If we choose the order so that the projection step
\begin{equation}\label{eq:mu_proj}
	\pi\gets \diag\Big(\frac{\mu}{\pi\one}\Big)\pi
\end{equation}
comes after \eqref{eq:simultaneous_update}, we can leave out the multiplication with $\diag(\mu)$ in \eqref{eq:simultaneous_update}, as this will cancel out in \eqref{eq:mu_proj}.
This results in Algorithm~\ref{alg:martOT}.
A dual algorithm can also be derived, but this has no apparent benefits, as the solution will not have a ``separable'' solution of the form $\diag(u)K\diag(v)$ because of the martingale constraint.


\begin{algorithm}[htb]
	\begin{algorithmic}
		\State \textbf{Input:} $\mu,\nu\in \R^M$, $c\in \R^{M\times M}$, $\eps >0$
		\State Set $A$ as in \eqref{eq:x_normalized}
		\State $\pi^{(0)}\coloneqq \exp(-c/\eps)$ 
		\While{not converged}
		\vspace{-\bigskipamount}
		\State \begin{align}
			\pi &\gets \pi\diag(\nu/ \pi^\tt \one) \\
			\pi &\gets \pi\odot \exp\Big(A^\tt \log\frac{A}{A\pi^\tt}\Big)^\tt \\
			\pi &\gets \diag(\mu/ \pi \one)\pi
		\end{align}
		\EndWhile
		\State \textbf{Output:} $\pi$
		\caption{Sinkhorn/GIS algorithm for martingale OT}
		\label{alg:martOT}
	\end{algorithmic}	
\end{algorithm}

\subsubsection{Numerical Example}\label{sec:martOT_example}

Next, we test Algorithm~\ref{alg:martOT} on a toy example.
To this end, we first construct an example with $\mu\preceq\nu$.
For a given $\mu\in \mathcal P(\R)$, $0\leq s,t$, define
\begin{equation}
	\nu \coloneqq \frac 1 2 \Big((T_{-s})_\# \mu + (T_t)_\# \mu\Big),
\end{equation}
where $T_a=(x\mapsto x+a)$.
Then we have that $\mu\preceq\nu$, since for any convex function $\varphi\colon \R\to \R$, it holds
\begin{equation}
	\int_\R \varphi(x)\dd\mu(x)
	\leq \int_\R \frac 1 2 (\varphi(x+s)+\varphi(x-t))\dd\mu(x)
	= \frac 1 2 \Big( \int_\R \varphi(x) \dd((T_{-s})_\# + (T_t)_\#) \mu \Big)
	= \int_\R \varphi\dd\nu.
\end{equation}
Note that this example is easily generalized to $n$ shifted distributions with weights $0<\lambda\in \Delta_n$, or even to infinitely many distributions, using Jensen's inequality.
For this example, we choose $M=100$ with a uniform grid $-1=x_1<\dots<x_M=1$, $c(x, y)=\exp(y-x)$, and we sample $\mu$ from $\mathcal N(0, 0.2^2)$ in the grid points $x_i$.
We construct $\nu$ as described above and normalize both measures to sum to one.
Note that we chose $c$ according to the conditions of \cite[Thm.~1.9]{martOT16BJ}, such that the optimal solution without entropic regularization will be concentrated on two graphs.
For Algorithm~\ref{alg:martOT}, we choose $\eps=0.002$ and terminate once the maximum absolute difference in any of the marginal constraints or in $\pi x = \mu\odot x$ is less than $10^{-5}$.
The unregularized solution is computed using the \texttt{linprog} function from Python's scipy package.
The problem and the results are shown in Figure~\ref{fig:martOT}.

We observe that the optimal solution is indeed the so-called ``curtain coupling'' concentrated on two graphs.
While this is qualitatively also observed in the regularized version, it is more blurred out as a result of the regularization.
On the other hand, this solution was obtained in only $4\%$ of the computation time.
\begin{figure}[htb]
	\centering
	\includegraphics[width=\textwidth]{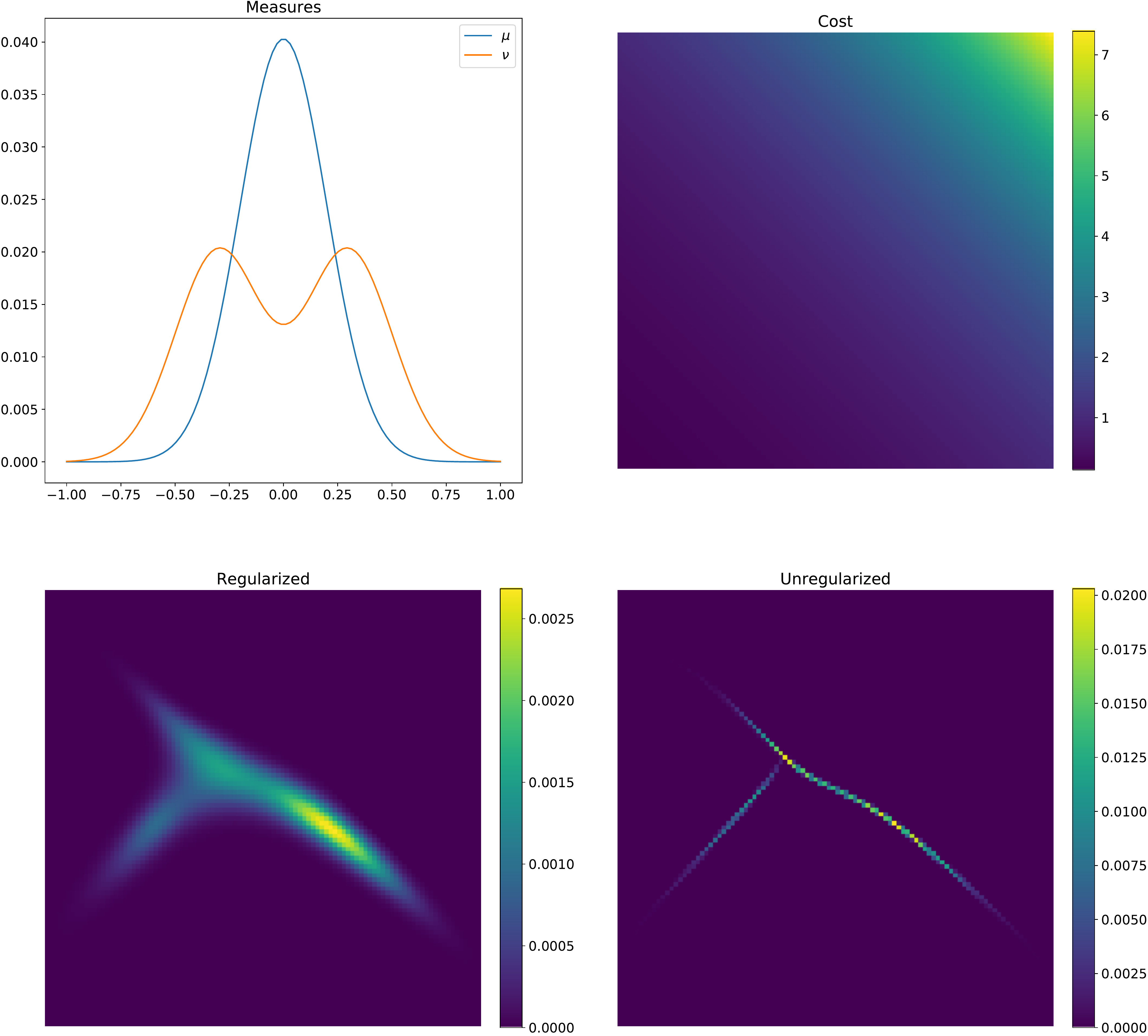}
	\caption{
		Top left: Measures $\mu$, $\nu$ in problem~\eqref{eq:martOT}.
		Top right: Cost function $c(x, y)=\exp(y-x)$.
		Bottom left: Approximation with entropic regularization computed by Algorithm~\ref{alg:martOT}.
		Bottom right: Unregularized solution computed using linear programming.}
	\label{fig:martOT}
\end{figure}

\subsection{Barycentric Weak Optimal Transport}

Next, we state the weak OT problem.
This notion of OT has originally been introduced in \cite{weakOTsolver17GRST} in search of conditions for concentration of measure results for discrete measures.
Nonetheless, it appears in many other topics such as martingale OT (see Section~\ref{sec:martOT})
or the causal OT problem \cite{causalOT17BBLZ} and its applications in mathematical finance \cite{wOTfinance20BBB},
since their constraints can be encoded in the cost function.
It has also been applied to OT barycenters \cite{weakBary21CTF}.

Let a cost function be of the form $C\colon \R^d\times \mathcal P(\R^d)\to \R_{\geq 0}\cup \{ +\infty \}$ and let $\pi_x$ denote the disintegration of $\pi$ with respect to $x\in X$.
Then the weak OT problem is stated as
\begin{equation}\label{eq:weakOT}
	V(\mu, \nu)\coloneqq \inf_{\pi\in \Pi(\mu, \nu)} \int C(x, \pi_x) \dd \mu(x).
\end{equation}
For $C(x, \pi_x) = \int c(x, x')\dd \pi_x(x')$ for some $c\colon \R^d\times \R^d\to \R$, this reduces to the classical Monge--Kantorovich formulation of optimal transport \eqref{eq:monge_kantorovich}.
Another special case is martingale OT from Section~\ref{sec:martOT}, when
\begin{equation}
	C(x, \pi_x) = \begin{cases}
	\int c(x, x')\dd \pi_x(x'), & \int x'\dd \pi_x(x') = x \\
	+\infty, & \text{otherwise}.
	\end{cases}
\end{equation}
Here we consider the special case of barycentric weak OT in the discrete case with $\supp(\mu)$, $\supp(\nu) \subset X = \{ x_1, \dots, x_M\}$.
Then we have
\begin{equation}
	C(x_i, \pi_{x_i}) = c\Big(x_i, \frac{1}{\mu_i} \sum_{j=1}^M \pi_{ij}x_j\Big),
\end{equation}
such that \eqref{eq:weakOT} reduces to
\begin{equation}\label{eq:weakOTbarycentric}
	V(\mu, \nu)
	\coloneqq \min_{\pi\in \Pi(\mu, \nu)} \sum_{i=1}^M \mu_i c\Big( x_i, \frac{1}{\mu_i} \sum_{j=1}^M \pi_{ij}x_j\Big)
	= \min_{\pi\in \Pi(\mu, \nu)} \sum_{i=1}^M \mu_i c(x_i, m_i)
\end{equation}
when substituting
\begin{equation}
m_i
\coloneqq \frac{1}{\mu_i} \sum_{j=1}^M \pi_{ij}x_j
= \sum_{j=1}^M \pi_{x_i} x_j
= \E_{\pi_{x_i}}[x].
\end{equation}
In what follows, we will assume that $c$ is convex in the second argument.
As usual, of particular interest is the cost function $c(x, x') = \Vert x-x'\Vert^2$.

Since $\pi$ appears as an argument of $c$ in \eqref{eq:weakOTbarycentric}, this problem is non-linear.
Thus, the strategy of applying entropic regularization to obtain a problem of the form \eqref{eq:problem} and apply Algorithm~\ref{alg:gis_pkl_mix} only works if we first relax it to a linear problem.
To this end, we rename $\pi$ to $\pi^x$ and introduce an auxiliary plan $\pi^y$ that will in some sense optimize over the $m_i$.
Moreover, we set an appropriate affine constraint to ensure that these $m_i$ fit to the original plan $\pi^x$ that fulfills the marginal constraints with respect to $\mu$ and $\nu$.
Let $Y\coloneqq \{ y_1, \dots, y_N \}$ with $X \subset\conv(Y)$.
As a first step towards the relaxation, consider the equivalent problem
\begin{align}\label{eq:weak_inflated}
	&\min_{\substack{\pi^x\in \R^{M\times M} \\ \pi^y\in \R^{M\times N}}} \sum_{i=1}^M \mu_i c\Big( x_i, \frac{1}{\mu_i} \sum_{j=1}^M \pi^x_{ij}x_j\Big)\quad
	\text{subject to} \quad \\
	& \pi^x, \pi^y\geq 0, \quad 
	\forall i: \sum_{j=1}^M\pi^x_{ij} = \sum_{k=1}^N \pi^y_{ik} = \mu_i,\quad
	\forall j: \sum_{i=1}^M\pi^x_{ij}=\nu_j,\quad
	\forall i: \sum_{j=1}^M \pi^x_{ij}x_j = \sum_{k=1}^N \pi^y_{ik} y_k.
\end{align}
The equivalence of \eqref{eq:weakOTbarycentric} and \eqref{eq:weak_inflated} is easy to verify:
Since we required $\conv(X)\subset \conv(Y)$, for any feasible plan $\pi^x$, there exists a plan $\pi^y$, such that for all $i=1, \dots, M,$
\begin{equation}
\sum_{k=1}^N \pi^y_{ik} = \mu_i
\quad \text{and}\quad
\frac{1}{\mu_i} \sum_{k=1}^N \pi^y_{ik} y_k
= \frac{1}{\mu_i} \sum_{j=1}^M \pi^x_{ij} x_j.
\end{equation}

Next, \eqref{eq:weak_inflated} is relaxed by employing Jensen's inequality: By the last constraint, it holds for the cost terms in \eqref{eq:weak_inflated} for any feasible solution that
\begin{equation}\label{eq:jensen_step}
	c\Big( x_i, \frac{1}{\mu_i} \sum_{j=1}^M \pi^x_{ij}x_j \Big)
	= 
	c\Big( x_i, \frac{1}{\mu_i} \sum_{k=1}^N \pi^y_{ik} y_k \Big)
	\leq \frac{1}{\mu_i} \sum_{k=1}^N \pi^y_{ik} c(x_i, y_k).
\end{equation}
Thus, we can state the relaxed problem as
\begin{align}\label{eq:weak_relaxed}
	&\min_{\substack{\pi^x\in \R^{M\times M} \\ \pi^y\in \R^{M\times N}}}
	\sum_{ik} \pi^y_{ik}c_{ik}\quad
	\text{subject to} \quad \\
	& \pi^x, \pi^y\geq 0, \quad 
	\forall i: \sum_{j=1}^M\pi^x_{ij} = \sum_{k=1}^N \pi^y_{ik} = \mu_i,\quad
	\forall j: \sum_{i=1}^M\pi^x_{ij}=\nu_j,\quad
	\forall i: \sum_{j=1}^M \pi^x_{ij}x_j = \sum_{k=1}^N \pi^y_{ik} y_k,
\end{align}
where $c_{ik}=c(x_i, y_k)$, such that it holds for any feasible solution that $V(\mu, \nu) = \eqref{eq:weak_inflated}\leq \eqref{eq:weak_relaxed}$.

Before adding entropic regularization and deriving the algorithm for this problem, we check that we can approximate \eqref{eq:weak_inflated} by \eqref{eq:weak_relaxed} by choosing $Y$ to be some fine enough approximation of $\conv(X)$, such that the optimal $m_i$ in \eqref{eq:weakOTbarycentric} can be closely approximated by points in $Y$.
Note that we cannot expect convergence of the plan $\pi^x$ itself when refining $Y$ in general, as the solution of \eqref{eq:weakOTbarycentric} is already not unique: Consider
\begin{equation}
	\mu = \frac 1 2 (\delta(1, 0)+\delta(-1, 0)), \qquad
	\nu = \frac 1 3 (\delta(0, 1) + \delta(0, 0) + \delta(0, -1)).
\end{equation}
The corresponding problem \eqref{eq:weakOTbarycentric} is clearly solved by many different transport plans.
However, if $c$ is strictly convex in the second argument, then the cost is strictly convex in the target means $m_i$, such that the optimal means are unique.
In this case, the approximate means converge to the optimal ones.
%
\begin{proposition}
	Let $U\supset \conv(X)$ be an open neighborhood.
	Let $Y_n=\{ y_1, \dots, y_{N(n)} \}$ be such that $U\subset \conv(Y_n)$ for all $n\in \N$ and $\dist(U, Y_n)\to 0$, where
	\begin{equation}
	\dist(A, B) \coloneqq \sup_{a\in A}\inf_{b\in B} \Vert a-b\Vert
	\end{equation}
	is a lower bound to the Hausdorff distance.
	Furthermore, let $c\colon \R^d\times \R^d\to \R$ be convex in the second argument, and denote by $(\hat\pi^{x, n},\hat\pi^{y, n})$ an optimal solution of \eqref{eq:weak_relaxed} with respect to $Y_n$, $n\in \N$ with corresponding means
	\begin{equation}
	\hat m_i^n \coloneqq \frac{1}{\mu_i} \sum_{j=1}^M \hat\pi^{x, n}_{ij}x_j = \frac{1}{\mu_i}\sum_{k=1}^N \hat\pi^{y, n}_{ik} y_k, \quad i=1, \dots, M.
	\end{equation}
	Then it holds
	\begin{equation}\label{eq:weak_convergence}
		\lim_{n\to \infty}\sum_{i=1}^M \mu_i c(x_i, \hat m_i^n)
		= \lim_{n\to \infty} \sum_{ik} \hat\pi^{y, n}_{ik}c_{ik}
		= V(\mu, \nu).
	\end{equation}
	Furthermore, if $c$ is strictly convex in the second argument, then the optimal $\hat m_i$ in \eqref{eq:weakOTbarycentric} are unique, and it holds for $i=1, \dots, M$ that
	\begin{equation}
		\lim_{n\to \infty} \hat m_i^n = \hat m_i.
	\end{equation}
\end{proposition}
\begin{proof}
	Let $\hat\pi$ be an optimal plan in \eqref{eq:weakOTbarycentric} and $(\hat\pi^{x, n},\hat\pi^{y, n})$ optimal in \eqref{eq:weak_relaxed} with respect to $Y_n$.
	Denote the corresponding means by
	\begin{equation}
	\hat m_i \coloneqq \frac{1}{\mu_i} \sum_{j=1}^M \hat\pi_{ij}x_j, \quad \text{and} \quad
	\hat m_i^n \coloneqq \frac{1}{\mu_i} \sum_{j=1}^M \hat\pi^{x, n}_{ij}x_j = \frac{1}{\mu_i} \sum_{k=1}^N \hat\pi^{y, n}_{ik}y_k, \quad
	i=1, \dots, M,
	\end{equation}
	respectively.
	We construct a set of feasible, not necessarily optimal plans $\tilde\pi^n \coloneqq (\tilde\pi^{x, n},\tilde\pi^{y, n})$ for \eqref{eq:weak_relaxed}, for which we will also have
	\begin{equation}
	\lim_{n\to \infty} \sum_{ik} \tilde\pi^{y, n}_{ik}c_{ik} = V(\mu, \nu),
	\end{equation}
	such that using \eqref{eq:jensen_step}, the assertion \eqref{eq:weak_convergence} follows from
	\begin{equation}
		V(\mu, \nu)
		= \sum_{i=1}^M \mu_i c(x_i, \hat m_i)
		\leq  \sum_{i=1}^M \mu_i c(x_i, \hat m_i^n)
		\leq \eqref{eq:weak_relaxed}
		= \sum_{ik} \hat\pi^{y, n}_{ik}c_{ik}
		\leq \sum_{ik} \tilde\pi^{y, n}_{ik}c_{ik}.
	\end{equation}
	To this end, for every $i=1, \dots, M$, take a sequence of point sets $\smash{(y_{i, k}^n)^{n\in \N}_{k=1, \dots, 2^d}\subset Y_n}$, such that for every $k=1, \dots, 2^d$, we have $y_{i, k}^n\to \hat m_i$ for $n\to \infty$, and 
	\begin{equation}
		\sum_{k=1}^{2^d} \lambda_{i, k}^n y_{i, k}^n = \hat m_i, \qquad
		\sum_{k=1}^{2^d} \lambda_{i, k}^n=1.
	\end{equation}
	We see that this is possible by the assumptions on $(Y_n)_{n\in \N}$ as follows:
	Consider a hypercube $Q_\eps$ with side length $\eps$ and center $m_i$.
	Then by assumption, we can choose $n\in \N$ large enough, such that there exist $y_{i, k}^n\in Y_n$, $k=1, \dots, 2^d$, close enough the corners of $Q_\eps$, such that $\hat m_i$ is a convex combination of the $y_{i, k}^n$.
	Then $y_{i, k}^n\to \hat m_i$ with $\eps\to 0$.
	
	Now let a transport plan $\tilde\pi^{y, n}$ for every $n\in \N$ be defined by
	\begin{equation}
		\tilde\pi^{y, n} \coloneqq \sum_{i=1}^M \mu_i \sum_{k=1}^{2^d} \lambda_{i, k}^n \delta(x_i, y_{i, k}^n)
	\end{equation}
	and set $\tilde\pi^n \coloneqq (\hat\pi, \tilde\pi^{y, n})$, which is feasible in \eqref{eq:weak_relaxed} by construction.
	Since $c$ is convex and hence continuous in the second argument, and since $\smash{y_{i, k}^n\to \hat m_i}$ for $n\to \infty$, for every $\delta>0$, we can choose $n\in \N$ large enough, such that for every $i=1, \dots, M$, $k=1, \dots, 2^d$, it holds
	\begin{equation}
		\vert c(x_i, y_{i, k}^n) - c(x_i, \hat m_i)\vert < \delta.
	\end{equation}
	But then it holds	
	\begin{align}
		\Big\vert \sum_{ik}\tilde\pi_{ik}^{y, n}c_{ik} - V(\mu, \nu) \Big\vert
		&= \Big\vert \sum_{ik}\tilde \pi_{ik}^{y, n} c_{ik} - \sum_{i=1}^M \mu_i c(x_i, \hat m_i) \Big\vert \\
		&\leq \sum_{i=1}^M \mu_i \sum_{k=1}^{2^d} \lambda_{i, k}^n\vert c(x_i, y_{i, k}^n) - c(x_i, \hat m_i)\vert
		<\delta.
	\end{align}
	Thus, the first assertion follows.
	The second assertion follows directly from the first together with strict convexity of the cost in \eqref{eq:weakOTbarycentric} in the means $m_i$.
\end{proof}
\begin{remark}
	Applying Jensen's inequality to weak barycentric OT \eqref{eq:weakOTbarycentric} directly without the auxiliary plan $\pi^y$ just results in Monge--Kantorovich OT \eqref{eq:monge_kantorovich}, which is not an approximation.
\end{remark}

\subsubsection{Algorithm Derivation}

Next, we solve \eqref{eq:weak_relaxed} using entropic regularization as before.
Note that
\begin{equation}
	-E(\pi^x)
	= \sum_{i, j=1}^M \pi^x_{ij}(\log\pi^x_{ij}-1)
	= \sum_{i, j=1}^M (\pi^x_{ij}\log\pi^x_{ij}-\pi^x_{ij}+1) - M^2
	= \KL(\pi^x, \one_{M\times M}) - M^2,
\end{equation}
such that for $\eps>0$, $K_{ik} \coloneqq \exp(-c_{ik}/\eps)$, we consider
\begin{equation}\label{eq:weak_entropic}
	\argmin_{\substack{\pi^x\in \R^{M\times M} \\ \pi^y\in \R^{M\times N}}}
	\KL(\pi^y, K) - E(\pi^x)
	= \argmin_{\substack{\pi^x\in \R^{M\times M} \\ \pi^y\in \R^{M\times N}}}
	\KL\Big( \frac 1 2 \begin{bmatrix}
	\vec(\pi^x) \\
	\vec(\pi^y)
	\end{bmatrix},
	\begin{bmatrix}
	\vec(\one_{M\times M}) \\
	\vec(K)
	\end{bmatrix}\Big)
\end{equation}
subject to the constraints from before in \eqref{eq:weak_relaxed}, which is an information problem of the form \eqref{eq:problem}.

Next, we derive the iterations of Algorithm~\ref{alg:gis_pkl_mix} for our problem \eqref{eq:weak_entropic} at hand.
Note that the we rescaled the optimization variable by $\frac 1 2$ for it to be a probability distribution.
However, when rescaling all the right hand sides $b$ of all our constraints accordingly, we will see that this just rescales all algorithm iterates by $\frac 1 2$.
In particular, it will converge to half of the result in the same number of iterations, which is why we can drop this rescaling in the following for convenience.

Set $x\coloneqq (x_1, \dots, x_M)\in \R^M$ and $y\coloneqq (y_1, \dots, y_N)\in \R^N$ as row vectors, which is again straightforward to generalize to higher dimensions $d$.
As before, the marginal projections are given by scaling \eqref{eq:scaling}.
It only remains to derive the GIS iteration for the mean consistency constraint that links $\pi^x$ and $\pi^y$, which we rewrite as
\begin{equation}\label{eq:weak_constr_rewritten}
	0
	= \sum_{j=1}^M \frac{\pi^x_{ij}}{\mu_i} x_j - \sum_{k=1}^N \frac{\pi^y_{ik}}{\mu_i} y_k
	= \sum_{j=1}^M \bar\pi^x_{ij} x_j + \sum_{k=1}^N \bar\pi^y_{ik} (-y_k),
\end{equation}
where $\bar\pi^x$, $\bar\pi^y$ denote the row-normalized plans with $\sum_{j=1}^M \bar\pi^x_{ij} = \sum_{k=1}^N \bar\pi^y_{ik} = 1$, $i=1, \dots, M$.
In order to obtain the positivity and column stochasticity requirements for GIS, as usual we perform a reparametrization of the affine subspace.
To this end, let
\begin{equation}
\xi^x\coloneqq \min_j x_j, \qquad
\xi^y\coloneqq \min_k -y_k, \qquad
\xi' \coloneqq \max\{ \max_j x_j-\xi^x, \; \max_k -y_k -\xi^y \}
\end{equation}
and rewrite \eqref{eq:weak_constr_rewritten} as
\begin{equation}
-\frac{\xi^x+\xi^y}{\xi'}
= \sum_{j=1}^M \bar\pi^x_{ij} \cdot \frac{x_j-\xi^x}{\xi'} + \sum_{k=1}^N \bar\pi^y_{ik} \cdot \frac{-y_k-\xi^y}{\xi'}
\end{equation}
and
\begin{equation}
2+\frac{\xi^x+\xi^y}{\xi'}
= \sum_{j=1}^M \bar\pi^x_{ij} \cdot \Big( 1-\frac{x_j-\xi^x}{\xi'} \Big) + \sum_{k=1}^N \bar\pi^y_{ik} \cdot \Big( 1+\frac{y_k+\xi^y}{\xi'} \Big),
\end{equation}

%
%
%
%
%
which is for all $i=1, \dots, M$ more compactly written as
\begin{equation}
	A^x(\bar\pi^x)^\tt + A^y(\bar\pi^y)^\tt = b,
\end{equation}
with
\begin{equation}\label{eq:weak_constraint_mats}
	b = \begin{bmatrix}
	-(\xi^x + \xi^y)/\xi' \\
	2+(\xi^x + \xi^y)/\xi'
	\end{bmatrix},
	\qquad
	A^x = \begin{bmatrix}
	(x - \xi^x)/\xi' \\
	1 - (x - \xi^x)/\xi'
	\end{bmatrix},
	\qquad
	A^y = \begin{bmatrix}
	(-y-\xi^y)/\xi' \\
	1 + (y + \xi^y)/\xi'
	\end{bmatrix}.
\end{equation}
As mentioned above, while $A^x$, $A^y$ are column-stochastic, $b$ sums to two since $[\bar\pi^x_{i, :}, \bar\pi^y_{i, :}]^\tt$ sums to two for every $i=1, \dots, M$.
Setting
\begin{equation}
	z \gets \log\Big( \frac{b}{A^x(\bar\pi^x)^\tt + A^y(\bar\pi^y)^\tt} \Big),
\end{equation}
similar to the case with martingale OT in Section~\ref{sec:martOT_alg_derivation}, the GIS iterations for all $i=1, \dots, M$ can be done in parallel by computing
\begin{equation}
	\pi^x \gets \pi^x\odot \exp(z^\tt A^x), \quad
	\pi^y \gets \pi^y\odot \exp(z^\tt A^y).
\end{equation}
Furthermore, as with martingale OT, the factors $\bar\pi^x_{ij}/\pi^x_{ij} = \bar\pi^y_{ik}/\pi^y_{ik} = \mu_i$ cancel out if the $\mu$-update is performed after this GIS iteration.
Altogether, this leads to Algorithm~\ref{alg:weakOT}.



\begin{algorithm}[htb]
	\begin{algorithmic}
		\State \textbf{Input:} $\mu,\nu\in \R^M$, $c\in \R^{M\times N}$, $\eps >0$
		\State Set $b$, $A^x$, $A^y$ as in \eqref{eq:weak_constraint_mats}
		\State $\pi^x\gets \one_{M\times M}$
		\State $\pi^y\gets \exp(-c/\eps)$ 
		\While{not converged}
		\vspace{-\bigskipamount}
		\State \begin{align}
		\pi^x &\gets \pi^x\diag(\nu/ (\pi^x)^\tt \one_M) \\
		z &\gets \log(b) - \log( A^x(\pi^x)^\tt + A^y(\pi^y)^\tt ) \\
		\pi^x &\gets \pi^x\odot \exp(z^\tt A^x)\\
		\pi^y &\gets \pi^y\odot \exp(z^\tt A^y)\\
		\pi^x &\gets \diag(\mu/ \pi^x \one_M)\pi^x \\
		\pi^y &\gets \diag(\mu/ \pi^y \one_M)\pi^y
		\end{align}
		\EndWhile
		\State \textbf{Output:} $\pi^x$
		\caption{Sinkhorn/GIS algorithm for martingale OT}
		\label{alg:weakOT}
	\end{algorithmic}	
\end{algorithm}

\subsubsection{Numerical Example}

We use the example from Section~\ref{sec:martOT_example}.
We aim to compare our algorithm with the weak OT solver from the Python OT package \cite{flamary2021pot}\footnote{\url{https://pythonot.github.io/gen_modules/ot.weak.html}. Accessed: 2.12.2022.}.
This program solves the unregularized barycentric weak OT problem for the cost function $c(x, x') = \Vert x-x'\Vert^2$ using a conditional gradient scheme. 
For our proposed algorithm, it is sufficient to choose $Y=X$, and we take $\eps=10^{-10}$.
Although both algorithms optimize another functional, for sake of comparison, we choose to terminate both algorithms once the absolute change in the weak OT cost \eqref{eq:weakOTbarycentric} is less than $10^{-9}$, which is the default setting in the weak OT solver.
The proposed algorithm terminates in approximately half of the time compared to the POT solver.

\begin{figure}[htb]
	\centering
	\includegraphics[width=\textwidth]{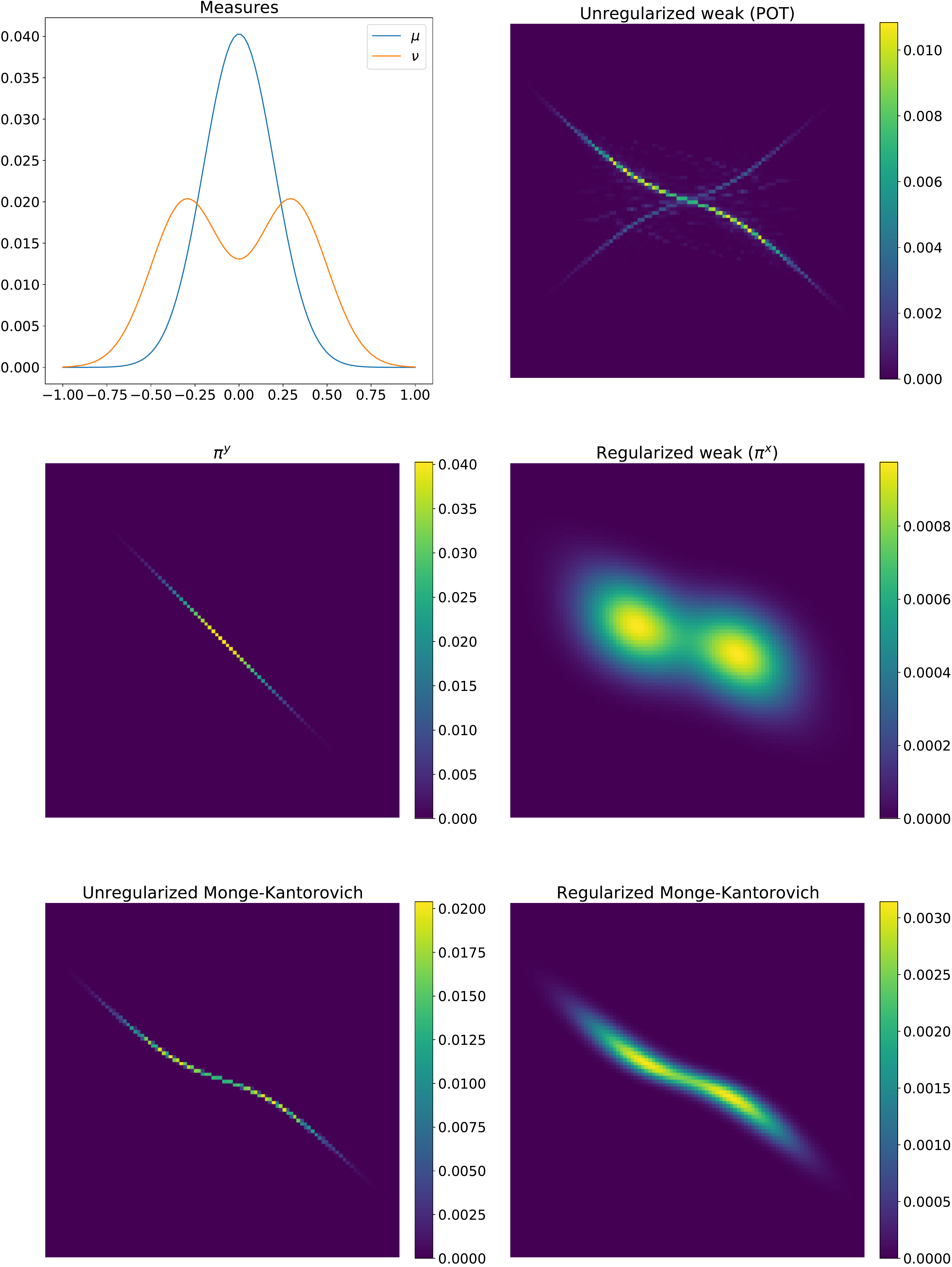}
	\caption{
		Top left: Measures $\mu$, $\nu$ in problem~\eqref{eq:weakOT}.
		Top right: Unregularized solution computed using POT solver.
		Middle left: Auxiliary plan $\pi^y$ computed by Algorithm~\ref{alg:weakOT}.
		Middle right: Regularized barycentric weak OT plan $\pi^x$ computed by Algorithm~\ref{alg:weakOT}.
		Bottom left: Standard OT plan with cost $c(x, y)=\Vert x-y\Vert^2$.
		Bottom right: As bottom left, with entropic regularization parameter $0.01$.
	}
	\label{fig:weakOT}
\end{figure}

The resulting transport plans are shown in Figure~\ref{fig:weakOT}.
We also show the transport plans from standard OT with and without regularization to highlight the differences, where we chose $0.01$ as the weight of the entropic regularization.
The auxiliary plan $\pi^y$ is almost exclusively supported on the diagonal, as can be expected from the problem \eqref{eq:weak_entropic} and the small choice of $\eps=10^{-10}$.
The weak OT cost of $\pi^x$ is around $5.15\cdot 10^{-8}$, whereas the POT solver's plan has a cost of $3.78\cdot 10^{-7}$.
Thus, while looking quite different qualitatively, both plans achieve a cost close to zero.
That is, the plans are very close to being martingale OT plans, despite the additional entropy term in \eqref{eq:weak_entropic}.
The unregularized weak OT plan preserves some of the features of the standard OT, having a lot of mass being almost supported on a graph.
Moreover, there is a faint reflected graph visible, transporting in the opposite direction in order to obtain means $m_i$ close to $x_i$, as well as some ``noisy'' artefacts.
On the other hand, the regularized plan is very smooth, looking somewhat like a mixture of almost isotropic Gaussians.
This can be explained as follows: Since there are many plans with the same optimal $m_i$ and hence the same cost, the regularization with $-E(\pi^x)$ in \eqref{eq:weak_entropic} will result in the smoothest possible one subject to the marginal and mean constraints.
Whether this is desirable or not, depends on the application: It loses the sparsity structure of the weak OT plan, but smoothes out the computational artifacts resulting from the ambiguity of the plan.
Whether this blur can be reduced by choosing the optimization objective differently without harming the barycentric weak OT cost \eqref{eq:weakOTbarycentric} is left for future research.
Moreover, it would be interesting to show whether adding the entropy term preserves the optimal means $m_i$ in general.


\subsection{Unbalanced Optimal Transport}\label{sec:unbalanced}

Suppose that we want to compute an optimal transport between measures with $\Vert \mu\Vert_1\neq \Vert \nu\Vert_1$.
This is the notation of so-called \emph{unbalanced optimal transport (UOT)}, which is relevant in real-world applications with noisy data.
Next to the approach of replacing the ``hard'' mass constraints given by \eqref{eq:marg_constraints} by a ``soft'' penalization with $\varphi$-divergences $D_\varphi$, also called $f$- or Csisz\'ar divergences, there is another approach with a linear objective and affine constraints called \emph{conic formulation} of UOT \cite{unbalanced18CPSV,unbalanced18LMS}.
They are equivalent for $D_\varphi=\KL$ and a certain choice of cost function, in which case the corresponding UOT distance is called the Hellinger--Kantorovich distance, which is a geodesic distance characterizing the weak-$*$ convergence on the positive measures.
We refer to \cite[Sec.~3]{HKbarys21FMS} for a concise collection of its different formulations.
It is proposed in \cite{toric22STVR} to solve a discretized conic formulation using GIS, which we only restate here in slightly modified form.

The central idea is to lift two given histograms $\mu,\nu\in \smash{\R^M_{\geq 0}}$
to row-stochastic matrices $\tilde \mu,\tilde\nu\in \smash{\R^{M\times N}_{\geq 0}}$ such that
\begin{equation}
	\sum_{k=1}^N ks\cdot \tilde\mu_{ik} = \mu_i \quad \text{and} \quad
	\sum_{k=1}^N\tilde\mu_{ik} = 1, \quad i=1, \dots, M,
\end{equation}
similarly for $\nu$, where $s>0$ is a fixed unit of mass.
Thus, an entry $\tilde\mu_{ik}$ is interpreted as a weight for having $k$ $s$-units of mass at location $i$.
This is clearly an over-parametrization: For example, $\tilde\mu_{i, :}=[2, 0]$ is equivalent to $\tilde\mu_{i, :}=[0, 1]$.
In the following, we drop the scaling constant $s$ by considering the rescaled $\mu/s,$ $\nu/s$ instead.

Using this idea, we formulate an unbalanced OT problem using a cost function $c\in \R^{(M\times N)\times(M\times N)}$,
where a value $c_{ikjl}$ is interpreted as the cost of generating $l$ units of mass at location $j$ from $k$ units of mass at location $i$.
Then we want to solve the linear OT problem
\begin{equation}\label{eq:uot}
	\min_{\pi\geq 0} \sum_{ikjl} c_{ikjl}\pi_{ikjl}
\end{equation}
subject to $\sum_{ikjl}\pi_{ikjl}=1$ and the marginal constraints
\begin{equation}\label{eq:uot_constraints}
	\sum_{j=1}^M\sum_{k=1}^K\sum_{l=1}^L k\cdot\pi_{ikjl} = \mu_i, \quad i=1, \dots, M, \qquad
	\sum_{i=1}^M\sum_{k=1}^K\sum_{l=1}^L l\cdot\pi_{ikjl} = \nu_j, \quad j=1, \dots, M.
\end{equation}
Couplings fulfilling \eqref{eq:uot_constraints} are called \emph{conic couplings}.
Feasibility is always guaranteed, since the coupling defined by
\begin{equation}
	\pi_{ikjl}\coloneqq \frac{\mu_i}{\Vert\mu\Vert_1}\delta(\Vert\mu\Vert_1, k)\cdot \frac{\nu_j}{\Vert\nu\Vert_1}\delta(\Vert\nu\Vert_1, l)
\end{equation}
is conic.

For the numerical solution of \eqref{eq:uot}, as above, the authors in \cite{toric22STVR} add an entropy term to apply GIS.
To this end, they stack all constraints in \eqref{eq:uot_constraints} into one matrix (as opposed to using multiple blocks, see Section~\ref{sec:block_advantages}) and a normalization similarly as in Section~\ref{sec:gis}.
However, the issues regarding the increased computational complexity and memory requirements caused by the lifting approach are not addressed.


\section{Conclusion}\label{sec:conclusion}

In this paper, we showed how entropic regularization can be used for optimal transport problems with affine constraints through GIS.
To this end, we gave a slightly more general algorithm than RBI-SMART, which specifically fits the context of optimal transport.
Our convergence proof is an adaptation of Csisz\'ars interpretation of GIS \cite{geomintDR89C}.
We specialized this algorithm to several problems from optimal transport.

The strategy in this paper could be applied to several other problems as well, such as
\begin{itemize}
	\item discretized moment-constrained optimal transport problems as in \cite{momentcoulomb21ACEL}, that were introduced as an approximation to the multi-marginal optimal transport problem with Coulomb costs motivated by quantum chemistry,
	\item approximate computation of the atomic Wasserstein distance proposed in \cite{fourierOT20C} that controls the Fourier coefficients of both measures, or
	\item generalized barycenters with marginals living on different subspaces \cite{genWbary21DGS}.
	For a nearest-neighbor discretization of the Radon transform, this can be computed with iterative scalings as in \eqref{eq:scaling} as in \cite{IBP15BCC,tomographic17AABC}.
	These methods could be generalized by the proposed framework to, e.g., linear interpolation discretizations as in \cite{parallelRadiation88C}.
	We remark, however, that this application can also be tackled with the elegant multi-marginal approach proposed in \cite{genWbary21DGS}.
\end{itemize}



Beside these additional applications, several ways to generalize the methods in this paper would be interesting:

\begin{itemize}
	\item It seems straightforward to extend the solution of the minimization problem with moment constraints in Section~\ref{sec:min_problem_moments} to an entropic barycenter problem
	\begin{equation}
		\min_{\nu\in \C(A, b)} \sum_{i=1}^N \lambda_i \OT_\eps(\mu_i, \nu)
	\end{equation}
	subject to the moment constraints, for given measures $\mu_i$, $i=1, \dots, N$ and $\lambda\in \Delta_N$.
	The proposal above is the special case for $N=1$.
	A multi-marginal formulation of this barycenter problem as in \cite{tree21HRCK, UMOT21BLNS} also seems possible.
	
	
	\item As the $\KL$ projection onto a half space lies on the boundary hyperplane, the above could be extended to the case with inequality constraints.
	
	\item Interesting would be a generalization to even more general, convex constraints.
	However, we are not aware of any analog of GIS for this case.
	As a possible application, in Section~\ref{sec:min_problem_moments}, one could also constrain the circular variance
	\begin{equation}
		1-\vert \E[\e^{\imag x}]\vert
	\end{equation}
	or circular standard deviation
	\begin{equation}
		\sqrt{\log(1/(1-\vert \E[\e^{\imag x}]\vert)^2)}
	\end{equation}
	on the torus.
	
	\item Since SMART also works in a more general, regularized setting \cite{smart93B}, it would be interesting to generalize the aforementioned algorithms to unbalanced OT, where the marginal constraints are relaxed to a penalization \cite{unbalancedScaling18CPSV,unbalanced18LMS,UMOT21BLNS}.
	
	
	\item Finally, it is natural to ask for an extension to the continuous setting.
	It is not clear how to do this in general.
	The proof in \cite{idivgeometry75C} for the convergence of iterative information projections uses compactness of a bounded set of probability distributions, which is only true for finite support sets.
	It also uses the Pythagorean identity $\KL(p, r) = \KL(p, q) + \KL(q, r)$, which is not true in general in the continuous realm.

\end{itemize}

\subsubsection*{Acknowledgements}

Many thanks to Florian Beier and Bernhard Schmitzer for fruitful discussions.

\bibliographystyle{abbrv}
\bibliography{literature}

\end{document}

%% file: preamble.tex
\usepackage[margin=2.5cm]{geometry}
\usepackage[utf8]{inputenc}
\usepackage{amsmath,amssymb,amsfonts,amsthm,mathtools}
\usepackage{graphicx}
\usepackage{caption,subcaption}
\usepackage{algorithm,algpseudocode,setspace}
\usepackage{color}
\usepackage{booktabs}
\usepackage{pifont}
\usepackage{url}
\usepackage{colonequals}
\usepackage{bbold} 
\usepackage{enumerate}
\usepackage[hidelinks]{hyperref}

\mathtoolsset{showonlyrefs}

\newtheorem{theorem}{Theorem}[section]

\newtheorem{remark}[theorem]{Remark}

\newtheorem{proposition}[theorem]{Proposition}

\makeatletter
\def\thm@space@setup{%
	\thm@preskip=\parskip \thm@postskip=0pt
}
\makeatother

\normalfont
\newcommand{\N}{\ensuremath{\mathbb{N}}}

\newcommand{\R}{\ensuremath{\mathbb{R}}}
\newcommand{\E}{\ensuremath{\mathbb{E}}}

\newcommand{\W}{\ensuremath{\mathcal{W}}}

\newcommand{\supp}{\textnormal{supp}}

\newcommand{\dd}{\,\mathrm{d}}

\renewcommand{\tt}{\mathrm{T}}

\newcommand{\KL}{\mathrm{KL}}

\newcommand{\eps}{\varepsilon}
\newcommand{\C}{\mathcal{C}}
\newcommand{\FFT}{\mathrm{FFT}}
\newcommand{\IFFT}{\mathrm{IFFT}}
\newcommand{\zero}{\mathbb{0}}
\newcommand{\one}{\mathbb{1}}
\renewcommand{\vec}{\mathrm{vec}}
\newcommand{\imag}{\mathrm{i}}
\newcommand{\e}{\mathrm{e}}
\newcommand{\conv}{\mathrm{conv}}
\newcommand{\id}{\mathrm{Id}}

\DeclareMathOperator*{\argmin}{argmin}

\DeclareMathOperator*{\diag}{diag}

\DeclareMathOperator{\dist}{dist}

\DeclareMathOperator*{\Var}{Var}

\DeclareMathOperator*{\Real}{Re}
\DeclareMathOperator*{\Imag}{Im}

\DeclareMathOperator{\OT}{OT}

\binoppenalty=\maxdimen
\relpenalty=\maxdimen

\numberwithin{equation}{section} 
\AtBeginDocument{
	\label{CorrectFirstPageLabel}
	
}

\newif\iflong